\makeatletter
\@namedef{subjclassname@2020}{\textup{2020} Mathematics Subject Classification}
\makeatother

\documentclass{amsart}
\usepackage{amsmath, amssymb, amsfonts, latexsym}
\usepackage{latexsym,rotate,eucal}

\newtheorem{theorem}{Theorem}[section]
\newtheorem{lemma}[theorem]{Lemma}

\theoremstyle{definition}

\newtheorem{remark}{Remark}

\allowdisplaybreaks[3]

\def\R{{\Bbb R}}

\numberwithin{equation}{section}

\begin{document}

\title{The   biharmonic hypersurface flow   and  the Willmore flow in  higher dimensions}

 \author{Yu Fu}
\address{Center of Applied Mathematics, Dongbei University of Finance and
Economics, Dalian 116025, P. R. China}\email{yufu@dufe.edu.cn}
\author{Min-Chun Hong}
\address{Department of Mathematics, The University of Queensland, Brisbane,
QLD 4072, Australia}
\email{hong@maths.uq.edu.au}
\author{Gang Tian}
\address{Beijing International Center for Mathematical Research,
 Peking University,
 Beijing, 100871
 P.R.China}
\email{gtian@math.pku.edu.cn}





\begin{abstract}
 The   biharmonic  flow of  hypersurfaces $M^n$
immersed in  the Euclidean space $\mathbb {R}^{n+1}$ for $n\geq 2$ is given by a fourth order
geometric
evolution equation, which is similar to
the Willmore flow. We apply the Michael-Simon-Sobolev inequality to  establish    new
Gagliardo-Nirenberg inequalities on hypersurfaces.  Based on these
Gagliardo-Nirenberg inequalities,   we apply   local energy estimates to extend the solution  by a  covering argument  and obtain an estimate on the  maximal  existence time  of  the  biharmonic  flow of  hypersurfaces in  higher dimensions.
In particular, we
solve a problem in \cite{BWW} on the biharmonic hypersurface  flow for $n=4$. Finally, we apply our new approach  to prove global existence of
the Willmore flow  in  higher dimensions.
\end{abstract}

\maketitle \markboth{Fu,  Hong and Tian} {The biharmonic hypersurface flow and  the Willmore flow }

\section{Introduction}
\hspace*{\parindent}

 In their  pioneering work \cite
 {ES},  Eells and Sampson
 introduced
 the  harmonic map    flow  to establish existence of a harmonic map,   representing   its  homotopy class between two Riemannian
 manifolds. Let   $M$ and $N$ be two Riemannian manifolds. A map $u: M\to N$ is called
  harmonic
 if its
 tension field  $\tau(u)$ vanishes; i.e.
\[\tau(u)= {\rm trace} \nabla d u=0.\]
Let  $u_0$ be a given smooth map from $M$ to $N$. One of the most important questions on harmonic maps is whether
 $u_0$  can be deformed to a harmonic map  in its homotopy
class $[u_0]$ or not. Assuming that the sectional curvature of the target manifold  $N$ is non-positive, Eells and Sampson \cite{ES} gave a positive
answer to  the above question by
  establishing the existence of the global smooth solution of the  heat flow for  harmonic maps as follows:
    \begin{equation}\label{HF}
    \partial_t u= \tau(u)=\triangle_M u+A(u)(\nabla u,\nabla u)
\end{equation}
 with initial value $u(0)=u_0$,  where $\triangle_M=-\nabla^*\nabla $ is the Laplace-Beltrami  operator  of $M$
 and $A(u)(\cdot,\cdot)$ is the second fundamental form of $N$.     Struwe \cite{st1}, \cite{st2} and Chen-Struwe \cite{CS} established many results on
  the
 existence and  partial regularity of weak solutions   to the harmonic map flow.   In two dimensions, Chang,   Ding and  Ye  \cite{CDY} showed an example that  the heat
 flow \eqref{HF} blows up in  finite time, so  in general   the heat flow for
 harmonic maps cannot have a global smooth solution.   Qing and Tian  \cite{QT} proved  an energy identity at
 the first singular time $T$  that facilitates   a   bubble tree as $t\to T$  through a finite number of  nontrivial harmonic maps on $S^2$      (see also \cite {Q},   \cite{P},  \cite{DT}).   Inspired by the harmonic
 map flow,  Hamilton  \cite{Ha} introduced  the Ricci flow, which was used by Perelman to settle the Poincar\'e conjecture. Therefore,  the
 fundamental work \cite
 {ES} of Eell-Sampson  has huge influences in differential geometry and analysis.

 To establish the variational theory in fibre bundles over manifolds,  Eells and Sampson \cite{ES1} also introduced a polyharmonic map of degree $r$ ({\it
 or $r$-harmonic map}), which  generalizes the concept of harmonic maps. In particular,   a lot of interesting results have been established in
 the   case of biharmonic maps with $r=2$.
 For a  map $u: M\to N$, its bienergy functional is given by
\begin{eqnarray}\label{bienergy}
E_2(u)=\frac{1}{2}\int_M|\tau(u )|^2dv_g.\nonumber
\end{eqnarray}
A map $u: M\to N$  is said to be {\it biharmonic} if it is
 a critical point of the bienergy functional (see, e.g., \cite{EL}); i.e. it satisfies
\begin{equation}\label{Bi Map1}
\tau_2(u)=\Delta_M\tau(u)-{\rm trace}\,
R^{N}(d u,\tau(u))du=0,
\end{equation}
where $\tau_2(u)$ is the bitension field of $u$ and $R^{N}$ is the curvature tensor of $N$. In their well-known survey \cite{EL},  Eells and
Lemaire
reemphasised the study of
 biharmonic maps between Riemannian manifolds.   Since then,  biharmonic maps have received a lot of attention in analysis.  Chang, Wang and Yang
\cite{CWY} proved  partial
regularity of stationary biharmonic maps.  For  more
results
on partial regularities of biharmonic maps, see \cite{W},  \cite{St3} and \cite {HY1}.

  In 1980s, B. Y. Chen  independently introduced  the concept of {\em biharmonic submanifolds}.
   Let  $M$  be a submanifold of a Riemannian manifold  $N$.
  $M$ is said to be a \textit{biharmonic submanifold}  if the isometric
 immersion $f: M\to N$ is a biharmonic map. Biharmonic submanifolds have attracted a lot of attention  in geometry.
  In particular, Chen \cite{Chen1991} proposed the  well-known biharmonic conjecture:
 \textit{Any biharmonic submanifold in  $\Bbb R^{n+1}$
is minimal}. Consider the special case that $M=M^n$ is a  hypersurface
 of $N=\Bbb R^{n+1}$.
 Chen \cite{Chen1991} and Jiang \cite{jiang1987} independently confirmed the conjecture for the case of $n=2$.  Hasanis and
Vlachos  \cite{Hasanis1995} showed that Chen's conjecture holds  in    $\Bbb R^4$ with $n=3$. In collaboration
with   Zhan
\cite{FHZ},   the first two authors   settled
Chen's conjecture  in $\Bbb R^5$ with $n=4$.  Recently, the first two authors with Zhan \cite{FHZ1} made   further progress to confirm
Chen's  conjecture in  $\Bbb R^{6}$ and the related Balmu\c{s}-Montaldo-Oniciuc conjecture for  $n=5$. Compared with the famous Bernstein problem  (see, e.g., \cite{Bombieri1969},
\cite{Simon-1968}),
Chen's conjecture in $\mathbb R^{n+1}$ with $n>5$ is a challenging  open problem. In  the past
 two decades,   significant  results  on biharmonic submanifolds related to  Chen's  conjecture have been established  (see, e.g.,
 \cite{OC20}, \cite{Fetcu}).

 For  an isometric
immersion $f: M^n\to
 \Bbb R^{n+1}$,  $ \Delta_{M} f=\vec H=H \nu$, where $\vec H$ is  the mean curvature vector field (the trace  of the second fundamental form),
 $H$ is the mean curvature and  $\nu$ is the outward normal  unit vector  field
of the hypersurface $M=f(M^n)$. Then  $f:
M^n\rightarrow \Bbb R^{n+1}$ is   biharmonic (c.f. \cite{Chen1991})  if and only if
\begin{equation}\label{Bi Map2}
\Delta_{M}^2f=\Delta_{M} \vec{H}=0.
\end{equation}
 It was known (see Lemma 2.1 in Section 2, or \cite{Chen1991}) that
	an isometric immersion $f: M^n\rightarrow \Bbb R^{n+1}$ is biharmonic if and only if it
	 satisfies the following two equations:
	\begin{equation} \label{biharmonic condition11}
	 \Delta_{M} H-H   |A|^2 =0, \quad
	2{\bf S}(\nabla H)+ H\nabla H=0,
 \end{equation}
where $A$ is the second fundamental form, ${\textbf{S} }$ is the   Weingarten operator defined by ${\bf S}(X)=-D_X\nu$ for any tangent vector field $X$ and $D$ denote the Levi-Civita
connections of $\mathbb R^{n+1}$.

Consider a closed,  immersed,  orientable hypersurface $M^n$
parametrized by $f_{0} : M^n\rightarrow \mathbb{R}^{n+1}$.
Motivated by the work of Eells-Sampson  \cite{ES}, we investigate whether $M^n$ can be deformed into a biharmonic hypersurface.
Therefore  we study the
  biharmonic flow of hypersurfaces
\begin{align} \label{E:theflow}
  &\frac{\partial f}{\partial t} =-\Delta_{M_t}^2 f=-\Delta_{M_t} {\vec{H}}= -\big(\Delta_{M_t}
H-H\vert A\vert ^2\big)\nu +2{\bf S}(\nabla
  H)+  H\nabla H
\end{align}
with initial value $f \left(  0 \right)=f_{0}$, where $\Delta_{M_t}$, $H$ and $\nu$ are respectively the Laplace-Beltrami
operator, the mean curvature and the  unit outward normal of the hypersurface
$M_{t} = f_{t} \left( M^n \right)$.

 Bernard,   Wheeler and  Wheeler \cite{BWW}  firstly investigated the heat  flow for biharmonic hypersurfaces with $2\leq n\leq 4$.
 In the case of $n=4$, they    obtained an estimate of the maximal existence of the heat  flow for biharmonic hypersurfaces by assuming
that the initial condition $f_0$ satisfies
 \begin{align} \label{Ad}\int_{M_0\cap B_{2R_0}\left( x \right)} (\left| A_{0} \right|^{4} +|\nabla A_0|^2 )d\mu_{0} \leq
 \varepsilon_{0} \end{align}
 for any $x\in \mathbb{R}^{n+1}$ and some fixed  $R_0>0$ with a small constant $\varepsilon_{0}$, where $A_0$ and $d\mu_0$ denote respectively the  second fundamental form  and the volume element of the  hypersurface $M_0$.

Comparing  the  assumption on the Willmore flow in \cite{KS1},  it is interesting whether one can remove  the additional condition on
$\int_{M_0\cap B_{2R_0}\left( x \right) } |\nabla A_0|^2
d\mu_{0} \leq \varepsilon_0 $ in \eqref{Ad}. In this paper,  we   solve the above problem   for $n=4$. More precisely, we have

\begin{theorem} \label{T1}
Let $f_0: M^n \rightarrow \mathbb{R}^{n+1}$ be a smooth  immersion for  $2\leq n\leq 5$.
Assume that  there is an absolute positive constant $\varepsilon_0 $ such
that
 \begin{align} \label{C1}
\int_{ M_0\cap B_{2R_0}(x)  } \left| A_{0} \right|^{n} d\mu_{0} \leq \varepsilon_{0}
 \end{align}
for any $x\in \mathbb{R}^{n+1}$ and some fixed  $R_0>0$.  Then the maximal
existence time $T$ of the solution on the flow \eqref{E:theflow} satisfies
$$T \geq  \delta R_0^4$$ for a small positive constant $\delta$.
Moreover, there exists a constant $\varepsilon> \varepsilon_0$ such that, for $0 \leq t \leq T$, we have
$$\int_{M_t\cap B_{R_0} (x)} \left| A(t) \right|^{n} d\mu \leq  \varepsilon\mbox{,}$$
where $M_t=f_t(M^n)$.
\end{theorem}

For the proof of Theorem    \ref{T1}, one of the key ideas is to  establish   local $L^2$-estimates under the condition \eqref{C1}.
When $n=2$, the proof  of  Theorem    \ref{T1} is  similar to the one for the
Willmore flow by  Kuwert and Sch\"atzle \cite{KS1},  applying the Michael-Simon Sobolev inequality to obtain local $L^2$-estimates under the condition \eqref{C1}.
Since the biharmonic map flow \eqref{E:theflow} is the flow  of higher dimensional hypersurfaces with $n>2$,   it becomes much more complicated
than the
Willmore
flow with $n=2$. Indeed, Bernard  et al.  \cite{BWW} needed an extra condition on $\nabla A_0$ in \eqref{Ad} for $n=4$.
 In our new proof of Theorem    \ref{T1}, we    apply  the Michael-Simon Sobolev inequality to derive a
Gagliardo-Nirenberg inequality  on hypersurfaces for $n\leq 5$.  For example, when $n=4$, for any $u\in
C^{1}_{c}\left( M^4 \right)$ satisfying that $\int_{[u\neq 0]}|H|^4\,d\mu\leq \varepsilon$ with sufficiently small $\varepsilon$, we have
\begin{align*}
 \int_{M^4} u^{4} d\mu
&\leq C  \int_{M^4} |\nabla^2 u|^2 \,d\mu   \int_{M^4} |  u|^2  \,d\mu.
\end{align*}
 With application of the extended  Michael-Simon Sobolev inequalities, we  obtain two local estimates on
$\|A\|^2_{L^2\left (f^{-1}\left(B_{R} (x) \right)\right)}$  and
$\|\nabla^2A\|_{L^2\left (f^{-1} (B_{R} (x))\right)}^2$ for any $R\leq \frac {R_0}4$ under the condition \eqref{C1}. Through a delicate
covering argument, we apply the
Gagliardo-Nirenberg inequality for $n\leq 5$ and   above local $L^2$-estimates to prove Theorem  \ref{T1}.
\medskip

Moreover, in the case of $n\geq 6$, we  have
\begin{theorem} \label{T2} Let $f_0: M^n \rightarrow \mathbb{R}^{n+1}$ be a smooth  immersion with  $n\geq 6$.  Under the same assumption
\eqref{C1},
the conclusion of  Theorem  \ref{T1}  also holds.
\end{theorem}
When $n\geq 6$, we cannot control $\|A\|_{L^n\left ( M\cap B_{R} (x)  \right)}$  by
$\|\nabla^2A\|_{L^2\left ( M\cap B_{R} (x) \right)}^2$, so  the proof of Theorem    \ref{T2} is much more complicated than  the proof of Theorem
\ref{T1}. More precisely,
we  establish  local Gagliardo-Nirenberg inequalities by higher order derivatives  of $A$; i.e.
 if   $n\geq 6$ is even, we have
\begin{align*}
\|A\|_{L^{n}( M\cap B_{R}(x)) }\leq C\| A\|^{\frac 2 {n} }_{L^2( M\cap B_{2R}(x)) }\sum_{m=0}^{\frac n2}\left \|\frac {\nabla^{m}
A}{R^{\frac
n2 -m}}\right \|^{\frac {n-2} {n}}_{L^2( M\cap B_{2R}(x)) }.
\end{align*}
However,   it is  difficult to obtain an $L^2$-estimate for $\nabla^{\frac n 2} A$ under the condition \eqref{C1} because $n\geq 6$.
To overcome this new difficulty, we need to modify the ideas in \cite{LSU} and \cite{KS1} to obtain the new Gagliardo-Nirenberg inequality as
following:
\begin{align}\label{H}
\|A\|_{L^{\infty}(  M\cap B_{R}(x))}\leq C\| A\|^{\frac 2 {n+2} }_{L^2(  M\cap B_{2R}(x))}\sum_{m=0}^{ \frac n 2 +1}
\left \|\frac {\nabla^m A}{R^{\frac
n2+1 -m}}\right \|^{\frac n
{n+2}}_{L^2(  M\cap B_{2R}(x)) }.
\end{align}
With the help of \eqref{H}, we apply  similar proofs as in \cite{KS1} to  obtain the energy estimate  for  $\nabla^{m} A$ for any integer $m\geq 1$.
Finally,  combining the generalized
Gagliardo-Nirenberg inequality  with local energy estimates, we apply  a more sophisticated covering argument  to prove Theorem  \ref{T2} for
any $n\geq 6$.

\begin{remark}  It is  known  that though a family  of diffeomorphisms,  \eqref{E:theflow} is equivalent to  a normal    biharmonic flow as in \cite{BWW} and \cite{PR}. A hypersurface satisfying  $2{\bf S}(\nabla H)+  H\nabla H=0$ is called biconservative and has a geometric meaning in better understanding  of  biharmonic submanifolds and Chen's biharmonic conjecture (see recent progress in \cite{Fetcu,OC20}). Palmurella and Rivi\`{e}re  \cite{PR}   emphasized the importance of the tangential component in the parametric Willmore flow. Moreover,  \eqref{E:theflow} is parabolic so it is nature to keep the tangential term in \eqref{E:theflow}.
\end{remark}

Since the Willmore
flow is  a fourth order
geometric
evolution equation analogous to the biharmonic flow, we  apply our new approach   in Theorem \ref{T2} to   prove the global existence of the Willmore flow in higher dimensions.

Let $M^2_{0}$ be a closed  immersed  orientable surface  parametrized by $f_{0} : M_0^2 \rightarrow \mathbb{R}^{N}$.  The two dimensional  Willmore flow of  surfaces
is given by the fourth order geometric evolution equation
 \begin{align} \label{W-flow}
  &\frac{\partial f}{\partial t}  = -  \left(\Delta^\perp {\vec{H}} + Q(A^0)  {\vec{H}}\right)
\end{align}
with initial value $f \left( 0 \right)=f_{0}$, where the Laplace operator $\Delta^\perp$ in \eqref{W-flow} is defined with respect to
the connection on normal vector fields along $f$, $A^0$ is the traceless part of  the second fundamental form  $A$ and
$Q(A^0)\vec H=A^0\left <\vec H, A^0\right > $.

Kuwert and Sch\"atzle \cite{KS1}  investigated the Willmore
flow  \eqref{W-flow}  and gave a
lower bound on the lifespan of a smooth solution.   Kuwert and Sch\"atzle \cite{KS2} further showed   the global existence of the Willmore
flow \eqref{W-flow} with   small initial data
as well as that the flow converges to a standard
round sphere (see also \cite{S}). Since then,  tremendous results  on the Willmore flow have been established. Recently, if the initial energy
is small,   Kuwert and Scheuer \cite{KS3}   proved stability estimates for the barycenter and the quadratic moment of the surface. Palmurella
and   Rivi\`{e}re \cite{PR} studied  the parametric  Willmore flow with a
tangential vector field $U$  so that $f(t, \cdot )$ is conformal for every $t$.

We investigate the Willmore flow in higher dimensions (i.e. $n\geq 2$).
Let $f: M^n\to \mathbb R^{n+1}$ be a closed hypersurface $M^n$ immersed in the Euclidean space $\mathbb R^{n+1}$.   We only consider the Willmore energy
\begin{align}\label{W-E}
{\mathcal W(f)}&=\frac 12 \int_{M} |A|^2 d\mu,
\end{align}
where $A$ is   the second fundamental form.
The Euler-Lagrange equation for the energy \eqref{W-E} is given by
\begin{align}\label{W-EL}
\Delta_{M} H-\frac{1}{2}\vert A\vert^2 H+C(A)=0,
\end{align}
where    $C(A):=g^{ij}g^{kl}g^{pq}A_{ik}A_{lp}A_{qj}$ with
summation over $i, j, k, l, p ,q\in \{1,2, \ldots, n\}$.
When $n=2$, it can be checked  that \[C(A)=\frac{1}{2}H(\vert A\vert^2+2\vert A^0\vert^2),\]  so the equation \eqref{W-EL} is  the exact  Willmore surface equation  in  \cite{KS2}:
$$\Delta_{M} H + H\vert A^0\vert^2 =0.$$
The Willmore flow for  hypersurfaces in higher dimensions
is defined by
\begin{align} \label{W-flow1}
  &\frac{\partial f}{\partial t}  = -\left(\Delta_{M_t} H-\frac{1}{2}\vert A\vert^2 H+C(A)\right) \nu
\end{align}
with initial value $f \left(  0 \right)=f_0$, where
$M_{t} = f_{t} \left( M^n \right)$.
The flow \eqref{W-flow1}  with $n\geq 3$ is  a generalization of the two dimensional  Willmore flow \eqref{W-flow}.
Then we have
\begin{theorem} \label{T4}
Let $f_0: M^n \rightarrow \mathbb{R}^{n+1}$ be a smooth  immersion with $n\geq 3$.
Assume that  there exists a  sufficient small   constant  $\varepsilon_0 $   such
that
 \begin{align} \label{WC1}
\int_{ M_0 \cap B_{2R_0}(x)  } \left| A_{0} \right|^{n} d\mu_{0} \leq \varepsilon_{0}
 \end{align}
for any $x\in \mathbb{R}^{n+1}$ and some fixed  $R_0>0$ and
\begin{align} \label{C2}
 \int_{M_0 }    |A_{0} |^{2} d\mu_{0} \leq C \varepsilon_0 R_0^{n-2}
 \end{align} with  some constant $C$.
  Then there is a global  smooth solution to the Willmore flow \eqref{W-flow1} that exists for all time $t\in [0,\infty)$.
  Moreover, as $t\to \infty$, the solution of the Willmore flow \eqref{W-flow1} converges to  a smooth solution of the Willmore equation \eqref{W-EL} locally in $C^k$ for all integers $k\geq 0$.
\end{theorem}

 \begin{remark}  Without  difficulties,   our Theorem  \ref{T2} and Theorem  \ref{T4} can be generalized to   both the biharmonic  flow and  the  Willmore  flow of  submanifolds in  higher dimensions and codimensions.
\end{remark}

The paper is organized as follows. In Section 2, we give some  basic estimates on  hypersurfaces and  compute the evolution
equations for elementary
geometric quantities.  In Section 3,  we apply the Michael-Simon Sobolev inequality to obtain two local energy inequalities. In Section 4, we
complete the proof of  Theorem \ref{T1}.
 In Section 5, we establish
Gagliardo-Nirenberg inequalities on hypersurfaces and prove  Theorem  \ref{T2}.  In Section 6, we study the Willmore flow in higher dimensions and complete the proof of  Theorem  \ref{T4}.

\section{Basic results of the  biharmonic flow} \label{S:notation}
In this section, we will outline  basic geometric quantities of  the  biharmonic hypersurface flow.

We consider a hypersurface $M=f(M^n)$ immersed in $\mathbb R^{n+1}$ via
$f: M^n\rightarrow \mathbb R^{n+1}$ and endow $M$ with a
Riemannian metric $g$ defined by

$$g_{ij}=g(\partial_i,\partial_j)=\langle
\frac{\partial f}{\partial x_i},
\frac{\partial f}{\partial x_j}\rangle=\langle \partial_i f, \partial_j f\rangle,$$
where $\langle\cdot, \cdot\rangle$ denotes the Euclidean scalar
inner product. The metric $g$ induces a natural isomorphism between
tangent and cotangent spaces. Let
$g^{ij}$ be   the inverse of $g_{ij}$.   The volume element $d\mu$ on $M$ is given in local coordinates by
$$ d\mu=\sqrt{\det g}~ dx. $$
 The Christoffel symbols $\Gamma^{k}_{ij}$ are defined by
\begin{align}
 \Gamma^{k}_{ij} =\frac{1}{2} g^{kl}\Big(\frac{\partial}{\partial x_i }g_{jl}+
 \frac{\partial}{\partial x_j }g_{il}-\frac{\partial}{\partial x_l }g_{ij}\Big).\nonumber
\end{align}
Recall that the Gauss and Weingarten formulas are given by
\begin{align}\label{GW}
\frac{\partial^2 f}{\partial x_i \partial
    x_j}=\Gamma_{ij}^k\frac{\partial f}{\partial x_k}+A_{ij}\nu,\quad
\frac{\partial}{\partial x_i} \nu=-{{}\bf S} (\frac{\partial f}{\partial x_i}),
\end{align}
where $A_{ij}$ are the components of the second fundamental form
$A$. Note that the Weingarten operator ${\bf S}$ and the second fundamental form $A$ are related by
\begin{align}\label{Wag}
{\bf S} (\frac{\partial f}{\partial x_i})=A_{ij}g^{jl}\frac{\partial
    f}{\partial x_l}.\end{align}

 We use the following notations for traces of the second
fundamental form $A$ on $M^n$:
\begin{align}
H={\rm trace}_g A=g^{ij}A_{ij},\quad |A|^2=g^{ij}g^{kl}A_{ik}A_{jl}.\nonumber
\end{align}
Then
\begin{align}
\nabla H= g^{ij}\frac{\partial
	H}{\partial x_i}\frac{\partial
	f}{\partial x_j}, \quad
{\bf S}(\nabla H)= g^{ij}g^{kl}\frac{\partial
	H}{\partial x_i}A_{jk}\frac{\partial
	f}{\partial x_l}. \nonumber
\end{align}
Moreover, the  components of the tracefree second fundamental form are defined by
\begin{align}
{A^0}=A-\frac{1}{n}Hg.\nonumber
\end{align}

The covariant derivative of a vector $X^i$ or a covector $Y_i$
on $M$ are defined respectively as
\begin{align}
\nabla_j X^i=\frac{\partial}{\partial x_j}X^i+\Gamma^i_{jk}X^k,\quad
\nabla_j Y_i=\frac{\partial}{\partial x_j}Y_i-\Gamma^k_{ij}Y_k.\nonumber
\end{align}
Recall that the Gauss and Codazzi equations are given
respectively by
\begin{align}\label{G}
R_{ijkl}=A_{ik}A_{jl}-A_{il}A_{jk},\\
\nabla_i A_{jk}=\nabla_j A_{ik}=\nabla_k A_{ij}.\label{C}
\end{align}

The following result gives the tangential and the normal parts of the
biharmonic immersion $\Delta_{M}^2 f$ for a hypersurface $M$ in the
Euclidean space $\mathbb R^{n+1}$.
\begin{lemma} \label{bp}
For an isometric immersion $f: M^n \rightarrow \mathbb{R}^{n+1}$,
we have
\begin{align}
\Delta_{M}^2 f&=\big(\Delta_{M}
H-H\vert A\vert ^2\big)\nu-\big(2{\bf S}(\nabla
  H)+H\nabla H\big),\nonumber
\end{align}
where $\Delta_{M}=-\nabla^*\nabla $ is  the Laplace-Beltrami
operator,  $H$ is the mean curvature and  $\nu$ is the unit outward  normal of $M$.
\end{lemma}
\begin{proof}Due to the well-known Beltrami  formula, we have
\[\Delta_{M}f=g^{ij}\left ( \frac{\partial^2 f}{\partial x_i \partial
    x_j}-\Gamma_{ij}^k\frac{\partial f}{\partial x_k}\right ) =H\nu.\]
In the local coordinates,
we obtain
\begin{align}
&\Delta_{M}^2 f= \Delta_{M}(H\nu )  =g^{ij}\Big\{\frac{\partial}{\partial x_i}\frac{\partial}{\partial
x_j}\big(H\nu\big)- \Gamma_{ij}^k\frac{\partial}{\partial
x_k}\big(H\nu\big) \Big\} \nonumber\\
& = (\Delta_{M} H)\, \nu +g^{ij}\Big\{\frac{\partial H}{\partial x_i}\frac{\partial}{\partial
	x_j}\nu+\frac{\partial H}{\partial x_j}\frac{\partial}{\partial
	x_i}\nu+H\frac{\partial}{\partial x_i}\frac{\partial}{\partial
x_j}\nu  -\Gamma_{ij}^k
H\frac{\partial}{\partial x_k}\nu  \Big\}. \nonumber
\end{align}
Using
the Gauss and Weingarten formulas \eqref{GW}, the above formula becomes
\begin{align}\label{Delta:1}
\quad \Delta_{M}^2f&=(\Delta_{M} H)\, \nu +g^{ij}\Big\{-\frac{\partial H}{\partial x_i}A_{jm}g^{ml}\frac{\partial f}{\partial x_l}-\frac{\partial
H}{\partial x_j}A_{im}g^{ml}\frac{\partial f}{\partial x_l}\\
&\quad-H\frac{\partial}{\partial
	x_i}\big(A_{jm}g^{ml}\frac{\partial f}{\partial
	x_l}\big)+H\Gamma_{ij}^k A_{km}g^{ml}\frac{\partial f}{\partial x_l}
\Big\}\nonumber \\
&=\big(\Delta_{M} H-H\vert A\vert ^2\big)\nu-2{\bf S}(\nabla H)\nonumber\\
&\quad-Hg^{ij}\Big\{\frac{\partial}{\partial
x_i}\big(A_{jm}g^{ml}\big)\frac{\partial f}{\partial
x_l}+A_{jm}g^{ml}\Gamma_{il}^k\frac{\partial f}{\partial
x_k}-\Gamma_{ij}^k A_{km}g^{ml}\frac{\partial f}{\partial x_l}
\Big\}\nonumber.
\end{align}
Since
$\frac{\partial }{\partial
x_i}g^{ml}=-g^{rl}\Gamma_{ir}^m-g^{mr}\Gamma_{ir}^l$, it follows
from the Codazzi equation that
\begin{align}
&g^{ij}\Big\{\frac{\partial}{\partial
x_i}\big(A_{jm}g^{ml}\big)\frac{\partial f}{\partial
x_l}+A_{jm}g^{ml}\Gamma_{il}^k\frac{\partial f}{\partial
x_k}-\Gamma_{ij}^k A_{km}g^{ml}\frac{\partial f}{\partial x_l}
\Big\}\nonumber\\
&=g^{ij}\Big\{\frac{\partial A_{jm}}{\partial
x_i}g^{ml}-A_{jm}g^{rl}\Gamma_{ir}^m-\Gamma_{ij}^k
A_{km}g^{ml}\Big\}\frac{\partial f}{\partial x_l}
\nonumber\\
&=g^{ij}\Big\{\frac{\partial A_{ij}}{\partial
x_m}g^{ml}+\Gamma_{ij}^k A_{km}g^{ml}-\Gamma_{mj}^k
A_{ki}g^{ml}-A_{jm}g^{rl}\Gamma_{ir}^m -\Gamma_{ij}^k A_{km}g^{ml}\Big\}\frac{\partial f}{\partial
x_l}\nonumber\\
&=g^{ij}\Big\{\frac{\partial A_{ij}}{\partial
x_m}g^{ml}-\Gamma_{mj}^k
A_{ki}g^{ml}-A_{jm}g^{rl}\Gamma_{ir}^m\Big\}\frac{\partial
f}{\partial x_l}\nonumber\\
&=\Big\{g^{ij}\frac{\partial A_{ij}}{\partial
x_m}+A_{ij}\frac{\partial g^{ij}}{\partial
x_m}\Big\}g^{ml}\frac{\partial f}{\partial
x_l}=\nabla H,\nonumber
\end{align}
which together with \eqref{Delta:1} completes the proof of Lemma \ref{bp}.
\end{proof}

Next, we compute the evolution equations for elementary
geometric quantities such as the metric, the volume element, the
second fundamental form and the mean curvature.
The covariant derivative  of mixed $(p, q)$ type tensor $T$ is given
by
\begin{align}
\nabla_k T^{i_1\ldots i_p}_{j_1\ldots j_q}=\frac{\partial }{\partial
x_k}T^{i_1\ldots i_p}_{j_1\ldots j_q}&+\Gamma_{km}^{i_1}T^{m\ldots
i_p}_{j_1\ldots j_q}+\cdots+\Gamma_{km}^{i_p}T^{i_1\ldots
m}_{j_1\ldots j_q}\nonumber\\
&-\Gamma_{kj_1}^{m}T^{i_1\ldots i_p}_{m\ldots
j_q}-\cdots-\Gamma_{kj_q}^{m}T^{i_1\ldots i_p}_{j_1\ldots
m}\nonumber.
\end{align}
The Laplacian of a tensor $T$ is defined by
\[\Delta_M T^{i_1\ldots i_p}_{j_1\ldots j_q}=-
\nabla^*\nabla  T^{i_1\ldots i_p}_{j_1\ldots j_q}.\]
From now on,    we denote $\Delta_M$ by $\Delta $.

  As in \cite{Ha} and \cite{KS1}, we will also write $T \star S$ to denote any tensor
formed by contraction on some indices of the tensors $T$ and $S$ by
the metric $g$ of $M_{t}$ and we use the notation $P_{r}^{m}\left( A
\right)$ to denote any combination of tensors of the type
$$P_{r}^{m}\left( A \right) = \sum_{i_{1}+ \cdots + i_{r}=m}c_{mr}
\nabla^{i_{1}} A \star \nabla^{i_{2}} A \star \cdots \star
\nabla^{i_{r}} A \mbox{,}$$ where the constants $c_{mr}$ are
absolute.

Let us recall the well-known Simons' identity \cite{Simon-1968}:
\begin{lemma} {\rm (Simons' identity). }
\begin{align}
\Delta A_{ij}=\nabla_{ij}H+HA_{il}g^{lm}A_{mj}-\vert A\vert ^2
A_{ij}.\nonumber
\end{align}
\end{lemma}

In the $`` \star "$ notation, Simons' identity becomes
\begin{align}
\Delta A=\nabla^{2}  H+A\star A\star A.\nonumber
\end{align}
Moreover, the Ricci identity  (see, e.g., \cite{KS1}) implies that for a mixed
tensor field $T$,
\begin{align}
\label{T}\nabla_{ij} T=\nabla_{ji} T+T\star A\star A.
\end{align}

\begin{lemma} \label{L1}
Let $f: M^n\times [0, T] \rightarrow \mathbb{R}^{n+1}$ be a solution to the flow $ \frac{\partial f}{\partial t}=F \nu+W$ with general $F$ and $W= a^l\frac {\partial f}{\partial x_l }$. Then we have
\begin{align}
\label{L21}\frac{\partial}{\partial t} g_{ij} &= -2F A_{ij}+\nabla_i a^l g_{jl}+\nabla_j a^lg_{il} \mbox{,}\\
\label{L22}\frac{\partial}{\partial t} g^{ij} &= 2F A^{ij}-\nabla_l a^i g^{jl}-\nabla_l a^jg^{il} \mbox{,}\\
\label{L23}\frac{\partial}{\partial t} \nu &=-\nabla F-a^l g^{ij}A_{il}\partial f_j \mbox{,}\\
\label{L24}\frac{\partial}{\partial t} d \mu &= (-HF+\nabla_l a^l)~d\mu  \mbox{,}\\
\label{L25}\frac{\partial}{\partial t} A_{ij}&= \nabla_{ij}F-F A_{jk}A_{il}g^{kl}+
\nabla_i a^lA_{jl}+\nabla_j a^lA_{il}+a^l\nabla_l A_{ij} \mbox{,}\\
\label{L26}\frac{\partial }{\partial t}H&=\Delta F +F \vert A\vert ^2+g^{ij}a^l\nabla_l A_{ij}.
\end{align}
\end{lemma}

\begin{proof}

Note $\frac{\partial}{\partial t}$ and
$\frac{\partial}{\partial x_i}$  commute. Then,  using \eqref{GW}  and  the fact that $\langle
\frac{\partial f}{\partial x_i}, \nu \rangle =0 $, we obtain
\begin{align}
\frac{\partial}{\partial t}g_{ij}
&=\langle \frac{\partial }{\partial x_i}(F\nu+W),
\frac{\partial f}{\partial x_j}\rangle+\langle \frac{\partial f
}{\partial x_i},
\frac{\partial }{\partial x_j}(F\nu+W)\rangle\nonumber\\
&= F\langle \frac{\partial \nu }{\partial x_i},
\frac{\partial f}{\partial x_j}\rangle + F \langle  \frac{\partial f }{\partial x_i},
\frac{\partial \nu}{\partial x_j}\rangle +\langle \frac{\partial W}{\partial x_i},
\frac{\partial f}{\partial x_j}\rangle+\langle \frac{\partial
f}{\partial x_i}, \frac{\partial W}{\partial x_j}\rangle\nonumber\\
&=-2F A_{ij}+\frac{\partial a^l}{\partial
x_i}g_{jl}+\frac{\partial a^l}{\partial
x_j}g_{il}+a^l\Gamma_{il}^mg_{mj}+a^l\Gamma_{jl}^mg_{mi}\nonumber\\
&=-2F A_{ij}+\nabla_i a^l g_{jl}+\nabla_j a^lg_{il}.\nonumber
\end{align}
This gives \eqref{L21}.

Furthermore, we compute $\frac{\partial}{\partial t}g^{ij}$. Since
$g^{ik}g_{jk}=\delta^i_j$, it follows that
\begin{align}
0=\frac{\partial (g^{ik}g_{jk})}{\partial t}=\frac{\partial
g^{ik}}{\partial t}g_{jk}+g^{ik}\big(-2F A_{jk}+\nabla_j a^l
g_{kl}+\nabla_k a^lg_{jl} \big),\nonumber
\end{align}
and hence
\begin{align}
g^{jm}g_{jk}\frac{\partial g^{ik}}{\partial t}=\frac{\partial
g^{im}}{\partial t}&=-g^{jm}g^{ik}\big(-2F A_{jk}+\nabla_j a^l
g_{kl}+\nabla_k a^lg_{jl} \big)\nonumber\\
&=2F A^{im}-\nabla_j a^i g^{jm}-\nabla_k a^m g^{ik}. \nonumber
\end{align}
Renumbering the indices in the above equation gives \eqref{L22}.

Noting again that $\left <\nu, \frac{\partial f}{\partial x_i}\right >=0$, it  follows  from (2.8) and \eqref{GW}   that
\begin{align}
\frac{\partial}{\partial t} \nu &= \langle \frac{\partial}{\partial
t}\nu, \nu \rangle \nu
+\langle \frac{\partial}{\partial
t}\nu, \frac{\partial f}{\partial x_i}\rangle \frac{\partial
f}{\partial x_j}g^{ij}=-\langle \nu, \frac{\partial}{\partial
t}\frac{\partial f}{\partial x_i}\rangle \frac{\partial f}{\partial
x_j}g^{ij}\nonumber\\
&=-\langle \nu, \frac{\partial}{\partial x_i}(F\nu+W)\rangle
\frac{\partial f}{\partial x_j}g^{ij}
=-g^{ij}\frac{\partial F}{\partial x_i} \frac{\partial
f}{\partial x_j}- a^l g^{ij}A_{il}\frac{\partial f}{\partial x_j}\nonumber\\
&=-\nabla F-a^l g^{ij}A_{il}\frac{\partial f}{\partial x_j}.\nonumber
\end{align}
We consider the evolution of the volume element $d\mu$:
\begin{align}\label{dm}
\frac{\partial}{\partial t} d\mu&=\frac{\partial}{\partial t}
\big(\sqrt{\det g}~ dx\big)=\frac{1}{2}{\sqrt{\det g}}~
g^{ij}\frac{\partial}{\partial t}g_{ij}
~dx=\frac{1}{2}g^{ij}\frac{\partial}{\partial t}g_{ij}
~d\mu  \\
&=\frac{1}{2}g^{ij}(-2F A_{ij}+\nabla_i a^l g_{jl}+\nabla_j
a^lg_{il}) ~d\mu=(-HF+\nabla_l a^l)~d\mu.\nonumber
\end{align}

By using the Gauss  and Codazzi formulas \eqref{GW}, we calculate the evolution of the second fundamental form to obtain \eqref{L24} by
\begin{align}
&\quad \frac{\partial }{\partial t} A_{ij}=\frac{\partial }{\partial
t}\langle \frac{\partial^2
f}{\partial x_i \partial x_j}, \nu\rangle\nonumber\\
&=\langle \frac{\partial^2 }{\partial x_i \partial
x_j}(F\nu+W), \nu\rangle-\langle \frac{\partial^2 f}{\partial
x_i \partial x_j}, \nabla F+a^l g^{mn}A_{ml}\frac{\partial f}{\partial x_n}\rangle \nonumber\\
&=\frac{\partial^2 F}{\partial x_i \partial x_j}+F\langle
\frac{\partial^2 }{\partial x_i \partial x_j}\nu, \nu\rangle+\langle
\frac{\partial^2 }{\partial x_i \partial x_j}W,
\nu\rangle-\Gamma_{ij}^k\langle \frac{\partial f}{\partial x_k},
\nabla F+a^l g^{mn}A_{ml}\frac{\partial f}{\partial
x_n}\rangle\nonumber\\
&=\frac{\partial^2 F}{\partial x_i \partial
x_j}-F A_{jk}A_{il}g^{kl}-\Gamma_{ij}^k\frac{\partial F}{\partial x^k}-\Gamma_{ij}^k
a^lA_{kl}+ \langle \frac{\partial^2
}{\partial x_i \partial x_j}W, \nu\rangle\nonumber \\
&=\nabla_{ij}F-F A_{jk}A_{il}g^{kl}-\Gamma_{ij}^k
a^lA_{kl}+\frac{\partial a^l}{\partial x_j}A_{il}+\frac{\partial
a^l}{\partial x_i}A_{jl}+a^l\frac{\partial
 A_{jl}}{\partial x_i}+ a^l\Gamma_{lj}^k A_{ik}\nonumber\\
&=\nabla_{ij}F-F A_{jk}A_{il}g^{kl}+\nabla_i
a^lA_{jl}+\nabla_j a^lA_{il}+a^l\nabla_i A_{jl}, \nonumber
\end{align}
 where we used the fact that
\begin{align}
\langle \frac{\partial^2 }{\partial x_i \partial x_j}W,
\nu\rangle
&=\langle\frac{\partial }{\partial x_i}\Big(\frac{\partial
a^l}{\partial x_j}\frac{\partial f}{\partial
x_l}+a^l\Gamma_{lj}^k\frac{\partial f}{\partial
x_k}+a^lA_{jl}\nu\Big),\nu \rangle\nonumber\\
&=\frac{\partial a^l}{\partial x_j}A_{il}+ a^l\Gamma_{lj}^k
A_{ik}+\frac{\partial a^l}{\partial x_i}A_{jl}+a^l\frac{\partial
 A_{jl}}{\partial x_i}\nonumber.
\end{align}
Using the above equations, we may compute the evolution of the mean
curvature:
\begin{align}
\frac{\partial }{\partial t}H&=g^{ij}\frac{\partial }{\partial t}A_{ij}+
A_{ij}\frac{\partial }{\partial t}g^{ij}\nonumber\\
&=\Delta F +F \vert A\vert ^2+ g^{ij}\big (\nabla_i a^lA_{jl}+\nabla_j a^lA_{il}+a^l\nabla_l A_{ij} \big)\nonumber\\
&\quad+A_{ij}\big(-\nabla_l a^i g^{jl}-\nabla_l a^jg^{il}\big)\nonumber\\
&=\Delta F +F \vert A\vert ^2+g^{ij}a^l\nabla_l
A_{ij}.\nonumber
\end{align}
These complete the proof of Lemma \ref{L1}.
\end{proof}

For now on, we focus   on the biharmonic flow  $ \frac{\partial f}{\partial t}=F \nu+W$ with
\begin{align}
F&:=-\big(\Delta
H-H\vert A\vert ^2\big),\\
 W&:= 2{\bf S}(\nabla H)+H\nabla H= a^l\frac {\partial f}{\partial x_l },\\
 a^l&:= \Big (2g^{ij}g^{kl}\partial_iHA_{jk}+H g^{il}\partial_i H\Big ).
\end{align}
For a given point $p\in M $, we may choose normal coordinates at $p\in M$ such that
\[g_{ij}(p)=\delta_{ij},\quad \frac {\partial g_{ij}}{\partial x_k}=0,\quad \Gamma^{k}_{ij}(p)=0.\]
We proceed in our proofs by
using normal coordinates at $p$. Taking into account (2.6)-(2.8), the following result is needed.
\begin{lemma} \label{T:evlneqns}
    For $f: M^n\times [0, T] \rightarrow \mathbb{R}^{n+1}$ evolving by $ \frac{\partial f}{\partial t}=F \nu+W$, the evolution equation of the
    induced connection coefficients $\Gamma^{i}_{jk}$ is given by
\begin{align}
 \label{L2.4}\frac{\partial}{\partial t} \Gamma^{i}_{jk} =  P_2^3(A)+P_4^1(A)\mbox{.}
\end{align}
\end{lemma}
\begin{proof}

Note that the Christoffel symbols $\Gamma^{k}_{ij}$ are determined by the
metric:
\begin{align}
 \Gamma^{k}_{ij} =\frac{1}{2} g^{kl}\Big(\frac{\partial}{\partial x_i }g_{jl}+
 \frac{\partial}{\partial x_j }g_{il}-\frac{\partial}{\partial x_l }g_{ij}\Big).\nonumber
\end{align}
Using \eqref{L21}-\eqref{L22} and the ``$\star$'' notation,  we have at each $p\in M$
\begin{align}
\frac{\partial}{\partial t} \Gamma^{k}_{ij} &=\frac{1}{2}
\frac{\partial g^{kl}}{\partial t}\Big(\frac{\partial}{\partial x_i
}g_{jl}+
 \frac{\partial}{\partial x_j }g_{il}-\frac{\partial}{\partial x_l }g_{ij}\Big) \nonumber\\
 &\quad+\frac{1}{2} g^{kl}\Big(\frac{\partial}{\partial x_i }\frac{\partial }{\partial t}g_{jl}+
 \frac{\partial}{\partial x_j }\frac{\partial }{\partial t}g_{il}-\frac{\partial}{\partial x_l }\frac{\partial }{\partial
 t}g_{ij}\Big)\nonumber\\
&= \frac{1}{2} g^{kl}\Big\{\frac{\partial}{\partial x_i }\big(-2F A_{jl}+\nabla_j a^m g_{lm}+\nabla_l a^mg_{jm}\big)\nonumber\\
&\quad+
 \frac{\partial}{\partial x_j }\big(-2F A_{il}+\nabla_i a^m g_{lm}+\nabla_l a^mg_{im}\big)\nonumber\\
&\quad-\frac{\partial}{\partial x_l }\big(-2F A_{ij}+\nabla_i a^m g_{jm}+\nabla_j a^mg_{im}\big)\Big\}\nonumber\\
&= P_2^3(A)+P_4^1(A)\nonumber,
\end{align}
where we have used $$F=-\Delta H+H|A|^2$$ and $$\nabla_i a^m=\nabla (\nabla A\star A).$$
This completes the proof of Lemma \ref{T:evlneqns}.
\end{proof}

\begin{lemma}

Let $f: M^n\times [0, T] \rightarrow \mathbb{R}^{n+1}$ be evolved  by
    $ \frac{\partial f}{\partial t}=-\Delta^2 f=F \nu+W$. Then we have
\begin{align}\label{3.14}
\frac{\partial}{\partial t} A&=-\Delta^2 A+
P_3^2(A)+P_5^0(A),\\\label{3.15}
\frac{\partial}{\partial t} \nabla^{m} A&=-\Delta^{2} \nabla^{m}
A+P_{3}^{m+2}(A)+ P_{5}^{m}(A),\quad
m\in\mathbb{N^+}.
\end{align}
\end{lemma}

\begin{proof}
The proof of this lemma  is similar to Lemmas 6 and 7 in \cite{BWW}. Since the flow has a  tangential part $W$, we give the proof for completeness.
	
Using \eqref{T}, we compute that
\begin{align}
\nabla_{ijkl}   H&=\nabla_{ikjl}   H+\nabla(\nabla A\star A\star A)\nonumber\\
&=\nabla_{klij}   H+P_3^2(A).\nonumber
\end{align}
By applying Simons' identity and the above equation, we get
\begin{align}
\label{L251}
\nabla^{2} \Delta   H=\Delta\nabla^{2}  H+ P_3^2(A)=\Delta^2 A+P_3^2(A).
\end{align}
Similarly, after interchanging covariant derivatives $k$ times in the highest order term, we have
\begin{align}\label{ex1}
\nabla^{k} \Delta^2 A=\Delta^2\nabla^{k} A+ P_3^{k+2}(A).
\end{align}

Since $F=-\Delta H+H\vert A\vert ^2$, it follows from applying \eqref{L251}, \eqref{L25} that
\begin{align*}
\frac{\partial}{\partial t} A&= \nabla^{2} F+F\star A\star A+\nabla(\nabla A\star A)\star A \\
&=-\nabla^{2}(\Delta   H)+ P_3^2(A)+ P_5^0(A) \nonumber\\
&=-\Delta^2 A+ P_3^2(A)+ P_5^0(A).\nonumber
\end{align*}
This proves \eqref{3.14}.

We now proceed by induction on $m$ and start with $m=1$.

It follows from \eqref{L2.4},  \eqref{3.14} and \eqref{ex1}  that
\begin{align}
\frac{\partial}{\partial t} \nabla A&= \nabla \big (\frac{\partial}{\partial t} A\big)+\big(P_2^3(A)+P_4^1(A)\big)\star A \nonumber\\
&=-\nabla(\Delta^2 A)+ P_3^3(A)+ P_5^1(A)\nonumber\\
&=-\Delta^2\nabla A+ P_3^3(A)+ P_5^1(A).\nonumber
\end{align}

Assume that $m\geq 2$ and it holds for $m-1$; i.e.
\begin{align}\label{3.18}
 \frac{\partial}{\partial t} \nabla^{m-1} A=-\Delta^2 \nabla^{m-1} A+P_3^{m+1}(A)+P_5^{m-1}(A).
\end{align}
Then using \eqref{L2.4},  \eqref{3.14} and \eqref{3.18}, we obtain
\begin{align}\label{Lm}
\frac{\partial}{\partial t} \nabla^{m} A&=\nabla\frac{\partial}{\partial t}  \nabla^{m-1} A+ \big(P_2^3(A)+P_4^1(A)\big)\star \nabla^{m-1}A&\\
& =-\nabla  \big (\Delta^2 \nabla^{m-1} A\big)+P_3^{m+2}(A)+P_5^{m}(A)) \nonumber\\
&=-\Delta^2 \nabla^{m} A+P_3^{m+2}(A)+P_5^{m}(A)\nonumber.
\end{align}

\end{proof}

 For simplicity, we denote $M_t=f_t(M^n)$  by $M$.
\begin{lemma} \label{T:e2}
Suppose that $\eta \in C^{2}\left( M \times
\left[ 0, T\right] \right)$ and $s$ is an integer with $s\geq 4$. Let $A$ be  a solution to
the flow \eqref{E:theflow} in $M\times [0, T]$. Then
\begin{align}\label{Lmc}
  &\frac{d}{dt} \int_{M}  \left| \nabla^{m} A \right|^{2}  \eta^{s} d\mu + \int_{M}  \left| \nabla^{m+2} A \right|^{2}  \eta^{s} d\mu \\
  & \quad \leq s \int_{M} \left| \nabla^{m} A \right|^{2} \eta^{s-1} \frac{\partial  \eta}{\partial t} d\mu
  +C \int_{M} \left| \nabla^m A \right|^{2}  \eta^{s-4} \left( \left| \nabla  \eta \right|^{4} +  \eta^2 \left| \nabla^{2} \eta
  \right|^{2} \right) d\mu \nonumber\\
  &\qquad + \int_{M}  \eta^{s} \Big[ P_{3}^{m+2}(A)+ P_{5}^{m}(A) \Big] \star \, \nabla^{m} A \, d\mu \mbox{,}\nonumber
\end{align}
for each $m\in\mathbb{N^+}$, where $C=C\left( s\right)$.
\end{lemma}
\begin{proof}
 Using \eqref{L22}, \eqref{dm} and  \eqref{Lm}, we have
\begin{align*}
 \frac{d}{dt} \int_{M}   \left| \nabla^{m} A \right|^{2} \eta^sd\mu
  = &  \int_{M}   \left| \nabla^{m} A \right|^{2}  \frac{\partial \eta^s}{\partial t} d\mu  - 2\int_{M} \left < \Delta^2 \nabla^{m} A,
  \nabla^{m} A\right > \eta^s d\mu \\
  & + \int_{M}   \Big[ P_{3}^{m+2}(A)+ P_{5}^{m}(A)\Big] \star \nabla^{m} A \,\eta^s\, d\mu.
\end{align*}

Integrating by parts and changing derivatives by \eqref{T} yield
\begin{align*}
 &-\int_{M} \left < \Delta^2 \nabla^{m} A, \nabla^{m} A\right > \eta^s d\mu =\int_{M} \left < \nabla\Delta \nabla^{m} A, \nabla (\nabla^{m}
 A\eta^s )\right > d\mu \\
  & =-\int_{M} | \nabla^{m+2} A|^2 \eta^s d\mu + \int_{M} \left <  \nabla^{m+2} A, \nabla^{m} A\nabla^2 (\eta^s )\right >d\mu \\
  &\quad -\int_{M} \left <  \nabla^{m+2} A, \nabla^{m+1} A\nabla  (\eta^s )\right > d\mu
  -\int_{M} \left <\nabla ( A\star A\star \nabla^{m+1} A), \nabla^{m} A\right > \eta^s d\mu
  \\
  &\leq -\frac  12   \int_{M} | \nabla^{m+2} A|^2 \eta^s d\mu  + \int_{M}    P_{3}^{m+2}(A) \star \nabla^{m} A \,\eta^s\, d\mu \\
  & +C \int_{M} \left| \nabla^m A \right|^{2}  \eta^{s-4} \left( \left| \nabla  \eta \right|^{4} +  \eta^2 \left| \nabla^{2} \eta
  \right|^{2} \right) d\mu,
\end{align*}
where we used
\begin{align*}
 \int_{M} | \nabla^{m+1} A|^2 \eta^{s-2}|\nabla \eta|^2 d\mu
 &\leq C \int_{M} | \nabla^{m} A| |\nabla^{m+2} A|  \eta^{s-2}|\nabla \eta|^2 d\mu\\ &+C \int_{M} \left| \nabla^m A \right|^{2}  \eta^{s-4}
 \left( \left| \nabla  \eta \right|^{4} +  \eta^2 \left| \nabla^{2} \eta
  \right|^{2} \right) d\mu .
\end{align*}
This proves our claim.
\end{proof}

\section{Two local energy inequalities on $\nabla^2A$} \label{S:energy}

\newtheorem{e1}{Lemma}[section]
\newtheorem{e2}[e1]{Lemma}
\newtheorem{e3}[e1]{Lemma}
\newtheorem{e4}[e1]{Corollary}
In this section, we will establish an  extension of the Michael-Simon-Sobolev inequality to obtain two local energy inequalities on $M$.

At first, let us recall    the well-known
Michael-Simon Sobolev inequality in \cite{MS}: For any $v\in
C^{1}_{c}\left( M \right)$, we have

\begin{align}\label{MS}\left( \int_{M} |v|^{\frac {n}{n-1}} d\mu \right)^{\frac{n-1}{n}} \leq C \left( \int_{M} \left| \nabla v\right| d\mu +
\int_{M} \left| H \right| \left| v
\right| d\mu \right) \end{align}
 for a constant $C>0$.
Taking $v=|u|^{\frac {p(n-1)}{n-p}}$ with $1\leq p< n$ and using H\"older's inequality, we have
\begin{align}\label{MS2}
\left( \int_{M} |u|^{\frac {pn}{n-p}} d\mu \right)^{\frac{n-p}{n}} \leq C  \int_{M}\left( \left| \nabla u\right|^p + \left| H \right|^p \left|
u
\right|^p \right)d\mu   \end{align} for any $u\in
C^{1}_{c}\left( M \right)$.

 Applying  the Michael-Simon Sobolev inequality \eqref{MS2} with $p=2$, we obtain
\begin{lemma} \label{L4.1}  Let $\eta\in C^2(M; \Bbb R)$ be a function with compact support in $M$ and $n>2$.
Then there exists a constant $C$ such that for any $u\in C^2(M)$, we have
\begin{align}\label{E1.1}
   \int_{M}|A| |u| \eta^4 d\mu
    \leq & C \left ( \int_{[\eta >0] } \left| A \right|^{n} d\mu \right)^{1/n}\int_{M}\left (  |\nabla  u |+ |A|  |u|
     \right  )\eta^4 d\mu  \\
    &+ C \left ( \int_{[\eta >0]} \left| A \right|^{n} d\mu \right )^{1/n} \int_{M} |u|  \eta^3 |\nabla \eta|   d\mu,\nonumber\\
    \label{E1.2}
   \int_{M}|A|^{2}|u|^2\eta^4 d\mu
    \leq & C \left ( \int_{[\eta >0] } \left| A \right|^{n} d\mu \right)^{2/n}\int_{M}\left (  |\nabla  u|^2+ |A|^2 |u|^2
     \right  )\eta^4 d\mu  \\
    &+ C \left ( \int_{[\eta >0]} \left| A \right|^{n} d\mu \right )^{2/n} \int_{M} |u|^{2} \eta^2 |\nabla \eta|^2  d\mu.\nonumber
\end{align}
 \end{lemma}
\begin{proof}
Using H\"older's inequality and  \eqref{MS}, we obtain
\begin{align}\label{E32}
       &\int_{M}|A| |u| \eta^4 d\mu \leq \left ( \int_{[\eta >0] } \left| A \right|^{n} d\mu \right)^{1/n}
       \left ( \int_{M} |u   \eta^4|^{\frac {n}{n-1}} d\mu\right )^{\frac {n-1}n}\\
    &  \leq C\left ( \int_{[\eta >0] } \left| A \right|^{n} d\mu \right)^{1/n}
       \left ( \int_{M}|\nabla  u   | \eta^2 + |A|   |u|   \eta^4 \right )d\mu \nonumber \\
       &\quad + C \left ( \int_{[\eta >0]} \left| A \right|^{n} d\mu \right )^{1/n} \int_{M} |u|  \eta^3  |\nabla \eta|   d\mu\nonumber.
\end{align}

Using H\"older's inequality and  \eqref{MS2}, we obtain
\begin{align}\label{E32}
       &\int_{M}|A|^{2}|u|^2 \eta^4 d\mu \leq \left ( \int_{[\eta >0] } \left| A \right|^{n} d\mu \right)^{2/n}
       \left ( \int_{M} |u   \eta^2|^{\frac {2n}{n-2}} d\mu\right )^{\frac {n-2}n}\\
        & \leq C\left ( \int_{[\eta >0] } \left| A \right|^{n} d\mu \right)^{2/n}
       \left ( \int_{M}|\nabla ( u   \eta^2)|^{2}+ |H|^2  |u|^2  \eta^4 \right )d\mu \nonumber\\
    &  \leq C\left ( \int_{[\eta >0] } \left| A \right|^{n} d\mu \right)^{2/n}
       \left ( \int_{M}|\nabla  u   |^{2}\eta^4 + |A|^2  |u|^2  \eta^4 \right )d\mu \nonumber \\
       &\quad + C \left ( \int_{[\eta >0]} \left| A \right|^{n} d\mu \right )^{2/n} \int_{M} |u|^{2} \eta^2 |\nabla \eta|^2  d\mu\nonumber.
\end{align}
\end{proof}

As applications of Lemma \ref{L4.1}, we obtain
\begin{lemma} \label{L4.2}  Let $\eta\in C^2(M; \Bbb R)$ be a function with compact support in $M$.
Then,  for $n> 2$, there exists a constant $C$ such that
\begin{align}\label{E1}
    &\int_{M}( |A|^{2}|\nabla A|^2+ |A|^{6} )\eta^4 d\mu\\
    \leq & C \left ( \int_{[\eta >0] } \left| A \right|^{n} d\mu \right)^{2/n}\int_{M} ( |\nabla^{2}  A|^2 +|A|^2 |\nabla A|^2+ |A|^{6})
    \eta^4 d\mu  \nonumber \\
    &+ C \left ( \int_{[\eta >0]}  \left| A \right|^{n} d\mu \right )^{2/n} \int_{M} |  A|^{2} ( |\nabla \eta|^4+|\nabla^2 \eta|^2)  d\mu.
    \nonumber
\end{align}

 \end{lemma}
\begin{proof}
Using  \eqref{E1.2} with $u=|\nabla A|$ and integrating by parts, we obtain
\begin{align}\label{E3.1}
       \int_{M}|A|^2  |\nabla A|^2 \eta^4 d\mu \leq & C\left ( \int_{[\eta >0]} \left| A \right|^{n} d\mu \right )^{2/n}
        \int_{M}  ( |   \nabla^{2} A|^2 + |A|^2 |\nabla A|^2)  \eta^4  d\mu
        \\
       & + C \left ( \int_{[\eta >0]} \left| A \right|^{n} d\mu \right )^{2/n} \int_{M}|\nabla  A|^{2} \eta^2 |\nabla \eta|^2
       d\mu\nonumber.
\end{align}
In addition, integrating by parts and using Young's inequality, we note
\begin{align}\label{E3.2}
       & \int_{M}   |   \nabla  A|^2  \eta^2 |\nabla \eta|^2  d\mu \\
       =&-\int_{M}   \left <A,   \Delta  A\right > \eta^2 |\nabla \eta|^2  d\mu -2\int_{M}   \left <A,   \nabla A\right > \cdot (\eta \nabla
       \eta
       |\nabla \eta|^2 +\eta^2 \nabla \eta \cdot \nabla^2 \eta)\, d\mu\nonumber\\
        \leq & \frac 12   \int_{M}   |   \nabla  A|^2  \eta^2 |\nabla \eta|^2  d\mu + \delta \int_{M}|\nabla^2 A|^{2} \eta^4  d\mu
        \nonumber\\
        &+C(1+\frac 1{\delta} ) \int_{M} |A|^2 (|\nabla
        \eta|^4+\eta^2|\nabla^2 \eta|^2)
       d\mu  \nonumber
\end{align} for a small constant $\delta>0$.
This implies that
\begin{align}\label{E3.2.1}
        \int_{M}   |   \nabla  A|^2  \eta^2 |\nabla \eta|^2  d\mu
        \leq  \delta \int_{M}|\nabla^2 A|^{2} \eta^4  d\mu +   C(1+\frac 1{\delta} )\int_{M}  |A|^2 (|\nabla
        \eta|^4+\eta^2 |\nabla^2 \eta|^2)
       d\mu .
\end{align}
Combining \eqref{E3.2.1} with \eqref{E3.1} yields
\begin{align}\label{E2}
       \int_{M}|A|^2  |\nabla A|^2 \eta^4 d\mu \leq & C\left ( \int_{[\eta >0]} \left| A \right|^{n} d\mu \right )^{2/n}
        \int_{M}  ( |\nabla^{2} A|^2 + |A|^2 |\nabla A|^2)  \eta^4  d\mu
        \\
       & + C \left ( \int_{[\eta >0]} \left| A \right|^{n} d\mu \right )^{2/n} \int_{M}|\ A|^{2}(|\nabla \eta|^4+|\nabla^2\eta |^2)
       d\mu\nonumber.
\end{align}
Using     \eqref{E1.2} with $u=|A|^2$ and Young's inequality, we have
\begin{align}\label{E1}
    \int_{M}|A|^{6} \eta^4 d\mu
    \leq & C \left ( \int_{[\eta >0] } \left| A \right|^{n} d\mu \right)^{2/n}\int_{M} (|A|^2 |\nabla   A|^2+|A|^6)
     \eta^4 d\mu  \\
    &+ C \left ( \int_{[\eta >0]}  \left| A \right|^{n} d\mu \right )^{2/n} \int_{M} | A|^{4} \eta^2 |\nabla \eta|^2  d\mu\nonumber\\
    &\leq   C\left ( \int_{[\eta >0]} \left| A \right|^{n} d\mu \right )^{2/n}
        \int_{M}   |   (|A|^2 |\nabla   A|^2+|A|^6)  \eta^4  d\mu\nonumber
        \\
       & + C \left ( \int_{[\eta >0]} \left|A \right|^{n} d\mu \right )^{2/n} \int_{M}|A|^{2}|\nabla \eta|^4
       d\mu\nonumber.
\end{align}
This proves our claim.
\end{proof}
In this section,  we always assume $\eta = \tilde \eta \circ f$, where $\tilde
\eta$ is a $C^{2}$ function with compact support in $B_{2\rho}(x_0)$  of $\mathbb{R}^{n+1}$, where $\tilde \eta (x) =1$ for  $x\in
B_{\rho}(x_0)=1$.  This implies
$$\left| \nabla \eta\right| \leq \frac {\Lambda}{\rho} \mbox{ and } \left| \nabla^{2} \eta \right| \leq \Lambda \left(\frac 1 {\rho^2}+\frac {
\left| A \right|}{\rho}
\right)
\mbox{,}$$
where $\Lambda$ is a bound on the $C^{2}$ norm of $\tilde \eta$.

 For simplicity, we denote $M_t=f_t(M^n)$  by $M$.
\begin{lemma} \label{L4.3} Let $f$ be the smooth solution of the biharmonic  flow \eqref{E:theflow} in
$\left[ 0, T\right]$.
Suppose that
$$\varepsilon (t) = \sup_{0 \leq t \leq T} \left\| A \right\|^{n}_{L^n (\left[ \eta >0 \right])} \leq \varepsilon$$
with a sufficiently small enough $\varepsilon$.
Then,  for any $t\in
\left[ 0, T\right]$, we have
  \begin{align}
  &\int_{\left[ \eta = 1\right]} \left| A\right|^{2} d\mu + \frac{1}{4} \int_{0}^{t} \int_{\left[ \eta =1\right]}  \left| \nabla^{2}
  A \right|^{2}  d\mu \, d\tau \\
 & \leq \int_{\left[ \eta_{0} >0\right]} \left| A_{0} \right|^{2} d\mu_{0} + C \frac t {\rho^4}  \sup_{0 \leq t <T} \int_{\left[ \eta
 >0\right]} \left| A\right|^{2} d\mu  \mbox{.}\nonumber
   \end{align}
 \end{lemma}
\begin{proof} It follows from using  \eqref{Lmc}   with $m=0, s=4$  and \eqref{E:theflow} that
 \begin{align} \label{E:1}
  &\quad    \frac{d}{dt} \int_{M}  \left| A \right|^{2} \eta^{4} d\mu +  \int_{M}  \left| \nabla^{2} A \right|^{2} \eta^{4} d\mu
    \\
&  \leq  \int_{M} \left[P_{3}^{2}(A)+ P_{5}^{0}(A) \right] \star A \eta^{4} d\mu  + C \int_{M} \left| A \right|^{2}
|\frac {\partial \eta}{\partial t}| \left| \nabla \eta
  \right|\eta^3  d\mu \nonumber\\
  &\quad + C\int_{M}   |A|^2
(|\nabla  \eta|^4 +\eta^2|\nabla^2 \eta|^2 )\,d\mu\nonumber \\
&\leq \frac 14  \int_{M}|\nabla^2 A|^2\eta^4d\mu  + C  \int_{M}  \left
[   \left| A \right|^2   \left| \nabla A
  \right|^{2}  + |A|^6 \right]\eta^{4} d\mu  +  \frac C{\rho^4}\int_{\left[ \eta>0\right]} \left| A \right|^{2} d\mu
\mbox{,} \nonumber
\end{align}
where we used the fact that
\begin{align} \label{Gr1}
\int_{M} | A|^2 \eta^2 |\nabla^2 \eta|^2\leq C \int_{M}|A|^6  \eta^{4} d\mu + \frac C{\rho^4}  \int_{\left[ \eta>0\right]}|A|^2\,d\mu
 \end{align}
and
\begin{align} \label{Gr2}
C\int_{M}|\nabla A|^2\eta^2|\nabla \eta|^2\,d\mu\leq \frac 1 4\int_{M}|\nabla^2 A|^2\eta^4 +\frac C{\rho^4}  \int_{\left[
\eta>0\right]}|A|^2\,d\mu.
\end{align}

Using  Lemma \ref{L4.2} and choosing $\varepsilon$ to be sufficiently small, we obtain
  \begin{align} \label{E1}
   \frac{d}{dt} \int_{M}  \left| A \right|^{2} \eta^{4} d\mu + \frac{1}{4} \int_{M}  \left| \nabla^{2} A \right|^{2} \eta^{4} d\mu \leq
   \frac C{\rho^4} \int_{\left[
\eta>0\right]}|A|^2\,d\mu \mbox{.}
\end{align}
This proves our claim by integrating in $t$.
\end{proof}

As applications of Lemma \ref{L4.1}, we also have
\begin{lemma} \label{L4.4}
Denote
$$\varepsilon = \int_{[\eta >0] } \left| A \right|^{n} d\mu.$$
Then there exists a constant $C$ such that
\begin{align}\label{E1}
   &\int_{M}(  |A|^4 |\nabla^2 A|^2+|A|^2|\nabla^3A|^2  )\eta^4 d\mu\\
    \leq & C \varepsilon^{2/n}\int_{M}  \left (|\nabla^{4}  A|^2+ |A|^4 |\nabla^2 A|^2+ |A|^2|\nabla^3A|^2+ |\nabla A|^2 |\nabla^2
    A|^2\right  )\eta^4d\mu \nonumber\\
     &+  C \varepsilon^{2/n}\int_{M} |\nabla^2 A|^{2}   \left (|\nabla \eta|^4+\eta^2 |\nabla^2\eta|^2 )\right )  d\mu \nonumber
\end{align}
and
\begin{align}\label{2E1}
   &\int_{M}(  |\nabla A|^2 |\nabla^2 A|^2+|A| |\nabla^2A|^3)\eta^4 d\mu\\
    \leq &  C\varepsilon^{1/n}\int_{M} ( |\nabla  A|^2  |\nabla^{2} A |^2+|A|^2|\nabla^{3} A |^2+ |\nabla^4 A|^2+ |A|^4|\nabla^2 A|^2)  \eta^4
    d\mu\nonumber\\
       &+   C\varepsilon^{1/n}\int_{M}    |\nabla^2A|^2 ( |\nabla \eta|^4 + \eta^2|\nabla^2\eta|^2) d\mu.\nonumber
\end{align}
 \end{lemma}
\begin{proof}
Integrating by parts and Young's inequality, we have
\begin{align}\label{E3.1.1.1}
  & \int_{M}  |\nabla^{3} A|^2   \eta^2 |\nabla \eta|^2   d\mu = - \int_{M}  \left <\nabla^{2} A, \nabla (\nabla^{3} A    \eta^2 |\nabla
  \eta|^2 )\right >  d\mu\\
  &\leq C \int_{M}    (|\nabla^{4} A|^2    \eta^4 +|\nabla^{2} A|^2 |\nabla \eta|^4 )   d\mu- \int_{M}  \left <\nabla^{2} A, \nabla^{3} A
  \nabla (  \eta^2 |\nabla \eta|^2 )\right >  d\mu\nonumber\\
  &\leq \frac 12 \int_{M}  |\nabla^{3} A|^2   \eta^2 |\nabla \eta|^2   d\mu+C \int_{M}  (|\nabla^{4} A|^2    \eta^4+  |\nabla^{2} A|^2
  (|\nabla \eta|^4+\eta^2|\nabla^2\eta|^2 ) ) d\mu \nonumber.
 \end{align}
Choosing $u=|\nabla^3A| \eta^2$ in \eqref{E1.2}  and using \eqref{E3.1.1.1}, we obtain
\begin{align}\label{E3.1.1}
  \int_{M} |A|^2 |\nabla^{3} A|^2   \eta^4   d\mu \leq &C\varepsilon^{2/n} \int_{M} ( |\nabla^{4} A|^2+|A|^2 |\nabla^3A|^2 ) \eta^4   d\mu\\
  &+ C\varepsilon^{2/n} \int_{M}   |\nabla^{2} A|^2 (|\nabla \eta|^4+\eta^2|\nabla^2\eta|^2 )   d\mu \nonumber.
 \end{align}
 Choosing $u=|A||\nabla^2A| \eta^2$ in \eqref{E1.2} and using Young's inequality, we have
\begin{align}\label{E3.1.1}
  &\int_{M} |A|^4 |\nabla^{2} A|^2   \eta^4   d\mu\\
   \leq &C\varepsilon^{2/n} \int_{M} (|\nabla A|^2 |\nabla^{2} A|^2+|A|^2 |\nabla^3A|^2+ |A|^4
  |\nabla^{2} A|^2 ) \eta^4   d\mu \nonumber\\
  &+ C\varepsilon^{2/n} \int_{M}   |A|^2 |\nabla^{2} A|^2 \eta^2 |\nabla \eta|^2 d\mu \nonumber\\
  \leq &C\varepsilon^{2/n} \int_{M} (|\nabla A|^2 |\nabla^{2} A|^2+|A|^2 |\nabla^3A|^2+ |A|^4
  |\nabla^{2} A|^2 ) \eta^4   d\mu\nonumber\\
  & +  C\varepsilon^{2/n} \int_{M}   |\nabla^{2} A|^2 |\nabla \eta|^4    d\mu \nonumber.
 \end{align}
Similarly,  choosing $u=| \nabla A|^2 $ in \eqref{E1.2}  yields
\begin{align}\label{E3.1.1.2}
  \int_{M} |A|^2 |\nabla  A|^4   \eta^4   d\mu & \leq C\varepsilon^{2/n} \int_{M} (  |\nabla  A|^2 |\nabla^2  A|^2 + |A|^2   |\nabla  A|^4 )  \eta^4\,d\mu
  \\
   &\quad +C\varepsilon^{2/n} \int_{M}   |\nabla  A|^4     \eta^2|\nabla \eta|^2\, d\mu\nonumber.
\end{align}
Integrating by parts, we note
\begin{align}\label{E3.1.2}
       & \int_{M} |\nabla A|^2 |\nabla^2 A |^2 \eta^4  d\mu =-\int_{M} \left <  \nabla_k (|\nabla^2 A|^2\nabla_k A \eta^4),  A\right >  d\mu\\
       &=-\int_{M} \left <  (|\nabla^2 A|^2\nabla_k^2 A+  \nabla_k (  |\nabla^2 A|^2) \nabla_k A  \eta^4 +|\nabla^2 A|^2\nabla_k A \nabla_k \eta^4 ,
       A\right >  d\mu.  \nonumber
\end{align}

Then, using Young's inequality yields
\begin{align}\label{E3.1.3}
        & \int_{M} |\nabla A|^2 |\nabla^2 A |^2 \eta^4  d\mu\\
        \leq &\frac 1 2 \int_{M} |\nabla A|^2 |\nabla^2 A |^2 \eta^4  d\mu +C\int_{M}( |A|^2 |\nabla^{3} A|^2 +|A| |\nabla^2 A|^3)  \eta^4
       d\mu \nonumber \\
       &+C\int_{M} |A|^2 |\nabla^2 A |^2\eta^2 |\nabla \eta|^2 d\mu
      \nonumber\\
      &\leq \frac 1 2 \int_{M} |\nabla A|^2 |\nabla^2 A |^2 \eta^4  d\mu + C\int_{M}( |A|^2 |\nabla^{3} A|^2 +|A| |\nabla^2 A|^3 +|A|^4|\nabla^2 A|^2)  \eta^4
       d\mu \nonumber \\
       & +C\int_{M}  |\nabla^2 A |^2  |\nabla \eta|^4 d\mu
      \nonumber.
\end{align}

Using \eqref{E1.1} and Young's inequality, we derive
\begin{align}\label{E3.2}
& \int_{M} |A| |\nabla^2 A|^3  \eta^4   d\mu\leq  C\varepsilon^{1/n}\int_{M} (|\nabla (|\nabla^2A|^3
\eta^4)| +|A| |\nabla^2 A|^3  \eta^4 ) d\mu  \\
&\leq C\varepsilon^{1/n}\int_{M} ( |\nabla^{2} A |^2 |\nabla^{3} A|   \eta^4+|A| |\nabla^2 A|^3  \eta^4 + |\nabla^2A|^3\eta^3|\nabla \eta|)
d\mu.\nonumber
 \end{align}
Using Young's inequality and integrating by parts,  we have
 \begin{align}\label{E3.2.1}
& \int_{M}|\nabla^{2} A |^2 |\nabla^{3} A|   \eta^4 d\mu=-\int_{M}\left <\nabla  A , \nabla (\nabla^{2} A  |\nabla^{3} A|   \eta^4)\right >
d\mu \\
&\leq \int_{M} ( C|\nabla  A|  |\nabla^{2} A | |\nabla^4 A|- \left <\nabla A, |\nabla^{3} A|\nabla^3A\right>  )  \eta^4
d\mu\nonumber\\
& \leq  \int_{M} (C( |\nabla  A|^2  |\nabla^{2} A |^2+ |\nabla^4 A|^2)+ \left < A, \nabla (|\nabla^{3} A|\nabla^3A)\right>  )
\eta^4  d\mu\nonumber\\
&\quad  + \int_{M}  \left < A \nabla  \eta^4, |\nabla^{3} A|\nabla^3A)\right>
 d\mu\nonumber\\
& \leq C \int_{M} ( |\nabla  A|^2  |\nabla^{2} A |^2+ |\nabla^4 A|^2+  |A|^2 |\nabla^3 A|^2   )  \eta^4  d\mu \nonumber
\\
& \quad  +C \int_{M}   |\nabla^3 A|^2 \eta^2 |\nabla \eta|^2  d\mu. \nonumber
 \end{align}
 Similarly, we have
 \begin{align}\label{E3.2.2}
& \int_{M}   |\nabla^3 A|^2 \eta^2 |\nabla \eta|^2  d\mu   \\
\leq & C \int_{M}   |\nabla^4 A|^2\eta^4  d\mu
    +C \int_{M} |\nabla^2 A|^2 (|\nabla \eta|^4 +\eta^2 |\nabla^2\eta|^2)   d\mu  \nonumber
 \end{align}
and
  \begin{align}\label{E3.2.3}
  &\int_{M} |\nabla^2A|^3\eta^3|\nabla \eta| d\mu=-\int_{M}\left < \nabla  A,  \nabla (|\nabla^2A| \nabla^2A\eta^3|\nabla \eta| )\right >
  d\mu\\
  &\leq C\int_{M} |\nabla  A|   |\nabla^2A| |\nabla^3A|\eta^3|\nabla \eta| d\mu + C\int_{M} |\nabla  A|   |\nabla^2A|^2 (\eta^2|\nabla \eta|^2
  + \eta^3|\nabla^2\eta|) d\mu \nonumber\\
  &\leq C\int_{M} |\nabla  A|^2   |\nabla^2A|^2\eta^4 d\mu + C\int_{M}  |\nabla^3A|^2\eta^2|\nabla \eta|^2 d\mu\nonumber \\
  &+ C\int_{M}    |\nabla^2A|^2 ( |\nabla \eta|^4 + \eta^2|\nabla^2\eta|^2) d\mu .\nonumber
 \end{align}
Plugging \eqref{E3.2.1}-\eqref{E3.2.3} into   \eqref{E3.2}, we obtain
\begin{align}\label{E3.3}
  \int_{M} |A| |\nabla^2 A|^3  \eta^4   d\mu
   & \leq C \varepsilon^{1/n}\int_{M} ( |\nabla  A|^2  |\nabla^{2} A |^2+ |\nabla^4 A|^2+  |A|^2 |\nabla^3 A|^2   )  \eta^4  d\mu
\\
& \quad  +  C\varepsilon^{1/n}\int_{M}    |\nabla^2A|^2 ( |\nabla \eta|^4 + \eta^2|\nabla^2\eta|^2) d\mu. \nonumber
 \end{align}

 Combining \eqref{E3.1.3} with \eqref{E3.3},   we have
\begin{align}\label{E3.4}
        & \quad \int_{M} |\nabla A|^2  |\nabla^2 A |^2\eta^4  d\mu\\
       \leq &  C\varepsilon^{1/n}\int_{M} ( |\nabla  A|^2  |\nabla^{2} A |^2+|A|^2|\nabla^{3} A |^2+ |\nabla^4 A|^2+ |A|^4|\nabla^2 A|^2)
       \eta^4  d\mu\nonumber\\
       &+   C\varepsilon^{1/n}\int_{M}    |\nabla^2A|^2 ( |\nabla \eta|^4 + \eta^2|\nabla^2\eta|^2) d\mu.\nonumber
\end{align}
This proves our claim.
\end{proof}

\begin{lemma} \label{m=2}
 Suppose
$$\varepsilon (t) = \sup_{0 \leq t <T} \left\| A \right\|^{n}_{L^n (\left[ \eta >0 \right])} \leq \varepsilon$$
for some small enough $\varepsilon$ depending upon $\Lambda$.
Then, under the biharmonic flow \eqref{E:theflow}, for any $t\in
\left[ t_1, T\right)$ with $t_1\geq 0$,
\begin{align*}
  &  \int_{M}  \left| \nabla^{2} A(t) \right|^{2} \eta ^{4} d\mu + \frac 12 \int_{t_1}^T\int_{M}  \left| \nabla^{4} A
  \right|^{2} \eta^{4} d\mu \\
  & \quad \leq  \int_{M}  \left| \nabla^{2} A(t_1) \right|^{2} \eta ^{4} d\mu
  + \frac {C}{\rho^4}\int_{t_1}^T\int_{\left[ \eta >0 \right]} (\left| \nabla^{2} A \right|^{2}+|A|^6)
 d\mu \mbox{.}
\end{align*}

\end{lemma}
\begin{proof}
Using Lemma \ref{T:e2} with $m=2$, we have
\begin{align}\label{4.5.1}
  &\quad\frac{d}{dt} \int_{M}  \left| \nabla^{2} A \right|^{2} \eta ^{4} d\mu + \int_{M}  \left| \nabla^{4} A \right|^{2} \eta ^{4} d\mu \\
  &  \leq 4 \int_{M} \left| \nabla^2 A \right|^{2}  \eta ^{3} \frac{\partial \eta}{\partial t} d\mu +\int_{M}  \Big[ P_{3}^{4}(A)+
  P_{5}^{2}(A)
  \Big] \star \, \nabla^{2} A \,\eta^4 d\mu\nonumber\\
  &+C \int_{M} \left| \nabla^{2} A \right|^{2}  \left( \left| \nabla \eta\right|^{4} + \eta^2 \left| \nabla^{2} \eta
  \right|^{2} \right) d\mu \mbox{.}\nonumber
\end{align}
 Integrating by parts and using Young's inequality, we  get
\begin{align}\label{4.5.2}
  &\int_{M} \left| \nabla^2 A \right|^{3}  \eta ^{3} |\nabla \eta| d\mu =- \int_{M}\left <   \nabla  A,  \nabla (  |\nabla^2 A|  \nabla^2 A
  \eta ^{3} |\nabla \eta| )\right >d\mu\\
&\leq    C\int_{M} ( \left| \nabla^2 A \right|^{2}   |\nabla A|^2\eta ^4 +\left| \nabla^3 A \right|^2   \eta ^{2} |\nabla \eta|^2 )  d\mu
\nonumber\\
&+C\int_{M}      \left| \nabla^2 A \right|^2      (  |\nabla \eta|^4+\eta^2|\nabla^2\eta|^2)  d\mu
\mbox{.}\nonumber
\end{align}
Then, it follows from
  \eqref{E:theflow} and  \eqref{4.5.2} that
\begin{align}\label{4.5.3}
  &
4 \int_{M} \left| \nabla^2 A \right|^{2}  \eta ^{3} \frac{\partial \eta}{\partial t} d\mu\leq \int_{M} \left| \nabla^2 A \right|^{2}
(|\nabla^2A|+|A|^3+|\nabla A| |A|)\eta ^{3} |\nabla \eta| d\mu\\
&\leq C\int_{M} ( \left| \nabla^2 A \right|^{2}   |\nabla A|^2+|A|^4|\nabla^2 A|^2)\eta ^4   d\mu\nonumber\\
&+C\int_{M}    \left| \nabla^3 A \right|^2   \eta ^{2} |\nabla \eta|^2  +  \left| \nabla^2 A \right|^2  (  |\nabla \eta|^4 +  \eta^2 \left|
\nabla^{2} \eta
  \right|^{2}) d\mu
\mbox{.}\nonumber
\end{align}
Integrating by parts yields
\begin{align*}
\int_{M}    \left| \nabla^3 A \right|^2   \eta ^{2} |\nabla \eta|^2   d\mu &\leq C\int_{M} |\nabla^2 A|   \left| \nabla^4 A \right|    \eta
^{2} |\nabla \eta|^2 + |\nabla^2 A|   \left| \nabla^3 A \right| |\nabla ( \eta ^{2} |\nabla \eta|^2)|\, d\mu
\\
&\leq \frac 1 2 \int_{M}    \left| \nabla^3 A \right|^2   \eta ^{2} |\nabla \eta|^2   d\mu + \frac 1 {16C} \int_{M}    \left| \nabla^4 A
\right|^2   \eta ^{4}   d\mu  \\
&+ C\int_{M} \left| \nabla^2 A \right|^2  (  |\nabla \eta|^4+\eta^2|\nabla^2\eta|^2)\,  d\mu .
&
\mbox{.}
\end{align*}
Combining this with \eqref{4.5.3} yields
\begin{align}\label{4.5.4}
   \int_{M} \left| \nabla^2 A \right|^{2}  \eta ^{3} \frac{\partial \eta}{\partial t} d\mu &\leq C\int_{M} ( \left| \nabla^2 A \right|^{2}
   |\nabla A|^2+|A|^4|\nabla^2 A|^2)\eta ^4   d\mu \\
&+\frac 1 {8} \int_{M}    \left| \nabla^4 A \right|^2   \eta ^{4}   d\mu+ C\int_{M} \left| \nabla^2 A \right|^2  (  |\nabla
\eta|^4+\eta^2|\nabla^2\eta|^2)  d\mu\nonumber
\mbox{.}
\end{align}

Using Young's inequality, we have
\begin{align}\label{4.5.5}
  &\int_{M}     P_{3}^{4}(A)\star \,  \nabla^{2} A\, \eta ^{4}\, d\mu \\
   \leq &\int_{M} (|A|^2|\nabla^4 A| +|A| |\nabla A| |\nabla^3 A|  +|A||\nabla^2A|^2+|\nabla A|^2|\nabla^2A|)|\nabla^2A| \,\eta
   ^{4}d\mu\nonumber\\
  \leq & C \int_{M}(|A|^4|\nabla^2A|^2+|A|^2|\nabla^3 A|^2 +|A||\nabla^2A|^3 +|\nabla A|^2|\nabla^2A|^2 )\,\eta
  ^{4}d\mu\nonumber\\
  & +\frac 1 8\int_{M} |\nabla^4 A|^2\,\eta ^{4}d\mu\nonumber
\end{align}
and
\begin{align}\label{4.5.6}
  \int_{M}     P_{5}^{2}(A)\star \, \nabla^{2} A \eta^4\, d\mu &\leq C \int_{M} (|A|^4|\nabla^2 A|^2+ |A|^3|\nabla  A|^2|\nabla^2A| ) \eta^4\,
  d\mu\\
  &\leq C \int_{M} (|A|^4|\nabla^2 A|^2+ |A|^2|\nabla  A|^4  ) \eta^4\, d\mu.\nonumber
\end{align}
 Integrating by parts yields
\begin{align*}
 & \int_{M}     |A|^2|\nabla  A|^4    \eta^4\, d\mu= \int_{M}    \left < |A|^2|\nabla  A|^2 \nabla_k A, \nabla_k A\right >    \eta^4\, d\mu
  \\
  =&-\frac 12 \int_{M}    |\nabla  A|^2 |\nabla  |A|^2 |^2   \eta^4\, d\mu - \int_{M}   \left <|A|^2 \nabla_k (|\nabla  A|^2 \nabla_k A ),   A\right >   \eta^4\, d\mu\\
  &- 4\int_{M}   \left <|\nabla  A|^2  |A|^2 \nabla_k A,
    A\right >   \eta^3 \nabla_k  \eta \, d\mu \\
   \leq&\frac  1 2 \int_{M}     |A|^2|\nabla  A|^4    \eta^4\, d\mu + C \int_{M}   |A|^4  |\nabla^2 A|^2  \eta^4\, d\mu\\
  & \quad + C\int_{M}    |A|^4  |\nabla A|^2   \eta^2 |\nabla  \eta |^2\, d\mu.
\end{align*}
Then using Young's inequality, we have
\begin{align}\label{4.5.7}
\int_{M}     |A|^2|\nabla  A|^4   \eta^4\, d\mu\leq  C \int_{M}   |A|^4  |\nabla^2 A|^2  \eta^4\, d\mu
+C\int_{M}    |A|^6   |\nabla  \eta |^4\, d\mu.
\end{align}
Moreover, we
note that
\begin{align*}
 &\int_{M} \left| \nabla^{2} A \right|^{2}  \left( \left| \nabla \eta\right|^{4} + \eta^2 \left| \nabla^{2} \eta
  \right|^{2} \right) d\mu
\leq  C\int_{\left[ \eta >0 \right]} \left| \nabla^{2} A \right|^{2}  \left( \frac 1  {\rho^4}  + \eta^2 |A|^2 \frac 1{\rho^2}\right) d\mu\\
  & \leq C \int_{\left[ \eta >0 \right]}  \left( \left| \nabla^{2} A \right|^{2}  \frac 1  {\rho^4}  + \eta^4 |A|^4 |\nabla^2A|^2\right ) d\mu
  .
\end{align*}

Substituting \eqref{4.5.4}-\eqref{4.5.7} into \eqref{4.5.1}, using Lemma \ref{L4.4}  and choosing   $\varepsilon$  to be sufficiently small, we obtain
\begin{align}\label{4.5.8}
  \quad\frac{d}{dt} \int_{M}  \left| \nabla^{2} A \right|^{2} \eta ^{4} d\mu + \frac 12 \int_{M}  \left| \nabla^{4} A \right|^{2} \eta ^{4}
  d\mu \leq \frac C {\rho^4}\int_{\left[ \eta >0 \right] } (\left| \nabla^{2} A \right|^{2}  + |A|^6 ) d\mu \mbox{.}
\end{align}
 This completes the proof. \end{proof}

\section{Proof of Theorem  \ref{T1}}

In this section, we derive the Gagliardo-Nirenberg inequalities for $n\leq 5$ and local energy inequality to prove  Theorem  \ref{T1}.
The following short time existence result is standard for evolution
equations such as \eqref{E:theflow} (see, e.g., \cite{Baker}, \cite{HP}).
\begin{theorem} \label{T:shorttime}
  For any smooth initial immersion $f_{0} : M^n \rightarrow \mathbb{R}^{n+1}$, there exists a unique, smooth solution $f\left(\cdot,
  t\right)$ to the flow \eqref{E:theflow} on some time interval $\left[ 0, T\right]$ with initial value $f_{0}$.
\end{theorem}

Let $f(\cdot,t)$ be the smooth solution to the biharmonic flow \eqref{E:theflow} in $[0,  T]$, where $T$ is the maximal time such that
\begin{align} \label{M.0}\sup_{0 \leq t \leq T,\,
x\in\R^{n+1}} \int_{M_t\cap B_{R_0}(x)}   |A|^{n}\  d\mu \leq \varepsilon
  \end{align}
for some small enough $\varepsilon$.
Then, choose  two constants $\delta$ and $R$ such that $T= {\delta}{R^4}$, where $\delta$ is sufficiently small.
For any $x\in\R^{n+1}$, let  $\tilde
\eta$ be a $C^{2}$ function with compact support in $B_{2R}(x)$  of $\mathbb{R}^{n+1}$ for a  positive  $4R \leq R_0$. Set $\eta = \tilde
\eta \circ f$. Then
\begin{align}\label{Cut}
\left| \nabla \eta\right| \leq \frac {\Lambda}{R} \mbox{ and } \left| \nabla^{2} \eta \right| \leq \Lambda \left(\frac 1 {R^2}+\frac {
\left| A \right|}{R}
\right)
\mbox{,} \end{align}
where $\Lambda$ is a bound on the $C^{2}$ norm of $\tilde \eta$.

Using Lemma \ref{L4.3}, for any $t\in [0, T]$ and $R$ with $4R\leq R_0$, we obtain
  \begin{align} \label{f.0}
  &\int_{ M_t\cap B_R(x)} \left| A(t)\right|^{2} d\mu + \frac{1}{2} \int_{0}^{T} \int_{ M_t\cap B_R(x)}  \left| \nabla^{2}A \right|^{2}   d\mu
  \, dt \\
 & \leq \int_{ M_0\cap B_{2R}(x))} \left| A_{0} \right|^{2} d\mu_{0} + \frac C {R^4} \int_0^T \int_{ M_t\cap B_{2R}(x)} \left| A\right|^{2}
 d\mu \,dt  \mbox{.}\nonumber
   \end{align}
By using \eqref{f.0},  there exists a $t_1\in [\frac T 4, \frac T 2]$  such that
\begin{align} \label{f.0.2}
  &  \frac {\delta R^4} 4  \int_{ M_{t_1}\cap B_R(x)} \left| \nabla^{2}
  A (t_1) \right|^{2} d\mu    \\
 \leq  &C \int_{  M_0\cap B_{2R}(x)} \left| A_0 \right|^{2} d\mu_{0} +  \frac C {R^4} \int_0^T \int_{M_t\cap
 B_{2R}(x)} \left| A \right|^{2} d\mu dt\nonumber
 \mbox{.}
   \end{align}
Then, utilizing   \eqref{f.0}-\eqref{f.0.2} and Lemma \ref{m=2}, we obtain
\begin{align}\label{f4}
   &  \sup_{\frac T 2\leq t\leq T}\int_{  M_t\cap B_{\frac R 2}(x)) } R^4 \left| \nabla^{2} A(t)  \right|^{2}  d\mu\\
    & \leq  \int_{  M_{t_1}\cap B_{R}(x)} R^4 \left| \nabla^{2} A(t_1) \right|^{2}  d\mu
  +      \frac C {R^4}\int_{t_1}^T \int_{  M_t\cap B_{R}(x)}  R^2 \left| \nabla^{2} A  \right|^{2}  d\mu dt\nonumber \\
 & \leq \frac C{\delta}\int_{  M_0\cap B_{2R}(x)}  {\left| A_{0} \right|^{2}} d\mu_{0} + \frac C {\delta R^4} \int_0^T \int_{  M_t\cap
 B_{2R}(x)}
 {\left| A\right|^{2}}
 d\mu d t \nonumber
  \mbox{.}
\end{align}

  Note that
\[|\nabla^2(A\eta^2 )|^2\leq |\nabla^2 A|^2\eta^4 + |\nabla A|^2  \eta^2|\nabla \eta|^2  + |A|^2 (\eta^2|\nabla^2 \eta|^2 +|\nabla\eta |^4).
\]
It follows from using \eqref{Gr1}, \eqref{Gr2} and \eqref{Cut} that
\begin{align*}
       \int_{M_t}   |   \nabla  A|^2  \eta^2 |\nabla \eta|^2  d\mu
        &\leq    C\int_{M}|\nabla^2 A|^{2} \eta^4  d\mu
        + C \int_{M_t} |A|^2 (|\nabla
        \eta|^4+\eta^2|\nabla^2 \eta|^2)
       d\mu  \nonumber\\
       &\leq  C\int_{M_t }|\nabla^2 A|^{2} \eta^4  d\mu+ \frac C{R^4}\int_{\left[ \eta>0\right]} \left| A \right|^{2} d\mu.
\end{align*}
Similarly, it follows from H\"older's inequality and \eqref{Cut} that
\begin{align*}
       &\int_{M_t}  |A|^2 \eta^2|\nabla^2 \eta|^2 d\mu\leq  \int_{M_t}  |A|^4 \eta^2\frac 1{R^2}d\mu+ \frac C{R^4}\int_{\left[ \eta>0\right]}
       \left| A \right|^{2} d\mu \\
       &\leq  \int_{M_t}  |A|^6 \eta^4 d\mu + \frac C{R^4}\int_{\left[ \eta>0\right]} \left| A \right|^{2} d\mu
       \end{align*}
       Combining above two estimates with \eqref{5.1.3}, we have
       \begin{align}\label{4.12}
       \int_{M_t}|\nabla^2(A\eta^2 )|^2  d\mu\leq  \int_{M_t}  |\nabla^2 A|^{2} \eta^4 d\mu + \frac C{R^4}\int_{\left[ \eta>0\right]} \left| A
       \right|^{2} d\mu.
       \end{align}

 \begin{lemma} \label{Vo}
  Let $f(\cdot, t)$ be the smooth solution to the biharmonic flow \eqref{E:theflow} in $[0,  T]$ satisfying \eqref{M.0}, where $T=\delta
  R^4$.
Then,  for any $t\in \left[ 0, \delta R^4\right)$, we have
\begin{align*}
\sup_{x\in \R^{n+1}}\mu \left (  M_t\cap B_{R}(x)\right )\leq CR^n.
\end{align*}
\end{lemma}
\begin{proof}
Let $\eta$ be a cut-off function satisfying \eqref{Cut}.  Then
\begin{align} \label{5.1.1}
  & \frac d{d t} \int_{M_t}  \eta ^{4} d\mu = 4 \int_{M_t}  \eta ^{3}   \frac {\partial\eta}{\partial t} d\mu +\int_{M_t}  \eta ^{4} \frac \partial{\partial t} d\mu \\
  &  \leq    C\int_{M_t}  \eta ^{3}|\nabla \eta| (|\nabla^2A|+|A|^3+|A| |\nabla A|) \, d\mu  +C\int_{M_t}  \eta ^{4} (|A| |\nabla^2 A| + |\nabla
  A|^2) d\mu\nonumber\\
  &\leq  C(\int_{M_t}  \eta ^{2}|\nabla \eta |^2 d\mu )^{1/2}\left (\int_{M_t}  \eta ^4 (|\nabla^2A|^2+|A|^6+|A|^2 |\nabla A|^2) \, d\mu \right
  )^{1/2\nonumber}\\
  &+C \left (\int_{M_t}  \eta ^{4} |A|^2 d\mu\right )^{1/2} \left (\int_{M_t}  \eta ^{4}   |\nabla^2 A|^2  d\mu\right )^{1/2}.\nonumber
\end{align}
It follows from \eqref{f.0} with sufficiently small $\delta$   that
 \begin{align} \label{5.1.2}
\sup_{x\in \R^{n+1}} \int_{  M_t\cap B_{R}(x)}  |A(t)|^2 d\mu  \leq C\sup_{x\in \R^{n+1}} \int_{  M_0\cap B_{R}(x)} \left| A_{0} \right|^{2}
d\mu_{0}\leq C\varepsilon^{2/n} R^{n-2} \end{align}
for $t\in [0,\delta R^4]$.
Using \eqref{E1}, Lemma \ref{L4.4} for a sufficiently small $\varepsilon$ and \eqref{f.0}, we obtain
\begin{align} \label{5.1.3}
  & \int_0^T \int_{M_t}  \eta ^4 (|\nabla^2A|^2+|A|^6+|A|^2 |\nabla A|^2) \, d\mu d t \\
  &\leq C\int_0^T \int_{M_t}  \eta ^4 |\nabla^2A|^2 \, d\mu d\tau+ \frac C{R^4} \varepsilon^{2/n}\int_0^T \int_{  M_t\cap B_{2R}(x)} |  A|^{2}
  d\mu d t\nonumber\\
  &\leq C\sup_{x\in \R^{n+1}} \int_{  M_0\cap B_{2R}(x) )} \left| A_{0} \right|^{2} d\mu_{0}\leq C\varepsilon^{2/n} R^{n-2}.\nonumber
\end{align}
Integrating by parts and using Young's inequality,  we note
\begin{align}\label{5.1.4}
  \int_{M_t}  \eta ^{4}  |\nabla A|^2 d\mu   \leq \frac 1 2 \int_{M_t}  \eta ^{4}  |\nabla A|^2 d\mu + C\int_{M_t}  \eta ^{4}  |A| |\nabla^2 A|   +
  |\eta|^2 |\nabla \eta|^2  |A|^2  d\mu.
  \end{align}
Substituting \eqref{5.1.2}-\eqref{5.1.4}  to \eqref{5.1.1} and  integrating in $t$,  we get
\begin{align}
    \int_{M_t}  \eta ^{4} d\mu -\int_{M_0}  \eta_0 ^{4} d\mu_0
  &\leq  C \left [R^n\mu \left (  M_t\cap B_{2R}(x) \right )\right ]^{1/2}+CR^n \nonumber\\
  &\leq \frac 12 \sup_{x\in \R^{n+1}}\mu \left (  M_t\cap B_{R}(x)\right )+ CR^n.\nonumber
\end{align}
This completes the proof. \end{proof}

Assume that $R<\frac {R_0}4$.
   Let $\{B_{R/2}(x_i)\}_{i=1}^{\infty}$ be an open cover of $\R^{n+1}$ such that   each $x\in \R^{n+1}$, $B_{2R}(x)$ can be covered by a finite number $L$
   of $B_R(x_i)$.

\subsection{Proof of Theorem  \ref{T1} for $n=4$}
In this subsection, we assume that $n=4$.
Let $u\in
C^{1}_{c}\left( M_t \right)$ satisfy  that $\int_{[u\neq 0]}|H|^4\,d\mu\leq \varepsilon$ for a sufficiently small $\varepsilon$.
 Using the Michael-Simon Sobolev inequality   and the H\"older  inequality,  we have
 \begin{align*}
 \left( \int_{M_t} |u|^{4} d\mu \right)^{\frac 1 2 }
&\leq C \int_{M_t} \left| \nabla u\right|^2 d\mu +C
\left (\int_{[u\neq 0]}   \left| H \right|^4 \,d\mu\right )^{\frac 1 2}   \left( \int_{M_t} \left| u\right|^{4}  d\mu \right)^{\frac 1 4}\\
&\leq  C \int_{M_t} \left| \nabla u\right|^2 d\mu +\frac  12   \left( \int_{M_t} \left| u\right|^{4} d\mu \right)^{\frac 1 4}.
\nonumber
\end{align*}
 For a sufficiently small $\varepsilon$,  we obtain
\begin{align}\label{GN1}
 \int_{M_t} u^{4} d\mu  &\leq C \left( \int_{M_t}  \left| \nabla u\right|^2  \,d\mu\right)^{2} \leq C \left |\int_{M_t}  \left< \nabla^*\nabla u,
 u\right>
 \,d\mu\right |^2\\
&\leq C  \int_{M_t} |\nabla^2 u|^2 \,d\mu   \int_{M_t} |  u|^2  \,d\mu .
\nonumber
\end{align}
Note that \eqref{GN1}
is a type of Gagliardo-Nirenberg's inequality  for any $u\in
C^{1}_{c}\left( M_t \right)$ satisfying  that $\int_{[u\neq 0]}|H|^4\,d\mu\leq \varepsilon$.

Choosing $u=|A|\eta^2$ in \eqref{GN1}, noting the fact that $|\nabla |A\eta^2||\leq |\nabla (A \eta^2)|$  and  using \eqref{4.12} yield
\begin{align}\label{MS4.1}
 &  \int_{  M_t\cap B_{R}(x)}   |A|^{4}\  d\mu\leq C  \int_{M_t}  |\nabla^2 (A\eta^2 )|^2 \,d\mu   \int_{M_t}  | A|^2 \eta^4 \,d\mu  \\
 & \leq C \left ( \int_{  M_t\cap B_{2R}(x))}  |\nabla^2 A|^2 \,d\mu  +\frac C{R^4}\int_{  M_t\cap B_{2R}(x))} \left| A \right|^{2} d\mu
 \right )\int_{  M_t\cap B_{2R}(x)}  | A|^2  \,d\mu\nonumber.
\end{align}

Now we give  the proof of Theorem  \ref{T1} for $n=4$.
\begin{proof}[Proof of Theorem  \ref{T1}]
 Let $T$ be the maximal existence time of the solution  satisfying \eqref{M.0} for  $T=\delta R^4$ with $4R\leq R_0$.
 For each ball $B_{R_0}(x_0)$, there exists a finite number of balls $\{ B_{\frac R 2}(x_i)\}_{i=1}^L$ such that
$$B_{R_0}(x_0)\subset\bigcup_{i=1}^L B_{R/2}(x_i)\subset  \bigcup_{i=1}^L B_{4R}(x_i)\subset B_{2R_0}(x_0).$$
It follows from Lemma \ref{Vo} and  H\"older's inequality that
\begin{align}\label{MS4.2}
  \left ( \int_0^T \int_{  M_t\cap B_{4R}(x)} \left| A\right|^{2} d\mu dt \right )^2   \leq C TR^4  \int_0^T \int_{  M_t\cap B_{4R}(x)} \left|
  A (T)\right|^{4} d\mu dt.
\end{align}
Using \eqref{MS4.1}  yields
\begin{align}\label{MS4.3}
& \int_{  M_t\cap B_{R_0}(x_0)} |A (T)|^{4}d\mu\leq \sum_{i}  \int_{ M_t\cap B_{R/2}(x_i)} |A (T)|^{4}d\mu \\
  &\leq \sum_{i} \int_{  M_t\cap B_{R}(x_i) ) } (R^2|\nabla^2 A (T)|^2 + \frac 1{R^2} \left| A (T)
 \right|^{2}) d\mu\,  \int_{  M_t\cap B_{R}(x_i)  } \frac {|  A|^2}{R^2}  \,d\mu.\nonumber
\end{align}
It follows from \eqref{f.0} and   \eqref{f4} that
\begin{align}\label{MS4.3.1}
 &\quad \int_{  M_t\cap B_{R}(x_i) } R^2 |\nabla^2 A (T)|^2   d\mu\,  \int_{  M_t\cap B_{R}(x_i) }\frac  {|  A (T)|^2}{R^2}  \,d\mu\\
 &\leq\frac  C{\delta} \left (\frac 1 {R^2}\int_{  M_0\cap B_{4R}(x_i))}  \left| A_0 \right|^{2}  d\mu_{0}\right )^2 +   \frac C {\delta
 R^{2}} \left (\int_0^T \int_{ M_t\cap B_{4R}(x_i)} {\left| A\right|^{2}} d\mu d t\right )^2\nonumber \\
 &\leq   \frac  C{\delta}  \int_{  M_0\cap B_{4R}(x_i)}|A_0|^4 d\mu_{0}+\frac {C} {R^4} \int_0^T\int_{ M_t\cap B_{4R}(x_i)}|A|^4\,d\mu d t
  \nonumber.
\end{align}
Similarly, we have

\begin{align}\label{MS4.3.2}
 &\quad \left (   \int_{  M_T\cap B_{R}(x_i) }\frac  {|  A(T)|^2}{R^2}  \,d\mu\right )^2\\
 &\leq  C \int_{  M_0\cap B_{4R}(x_i)}|A_0|^4 d\mu_{0}+\frac {C} {R^4} \int_0^T\int_{  M_t\cap B_{4R}(x_i)}|A|^4\,d\mu dt
  \nonumber.
\end{align}
It follows from \eqref{MS4.3}-\eqref{MS4.3.2} that for each $t\in [0,T]$, we have
\begin{align}\label{MS4.4.4}
 & \int_{  M_T\cap B_{R_0}(x_0))} |A(T)|^{4}d\mu\\
 &\leq \sum_i\frac  C{\delta}  \int_{  M_0\cap B_{4R}(x_i)}|A_0|^4 d\mu_{0}+\frac {C} {R^4}  \sum_i\int_0^T\int_{
  M_t\cap B_{4R}(x_i)}|A|^4\,d\mu d t\nonumber \\
 &\leq   \frac  C{\delta} \int_{  M_0\cap B_{2R_0}(x_0))}|A_0|^4 d\mu_{0}+  \frac {C} {R^4} \int_0^T\int_{  M_t\cap B_{2R_0}(x_0)}|A|^4\,d \mu
 dt\nonumber\\
&\leq \frac {C\varepsilon_0}{\delta}+C\delta  \varepsilon\leq \frac 12 \varepsilon, \nonumber
\end{align}
where we choose that $C\delta\leq \frac 1 4$ and $\frac {C\varepsilon_0}{\delta}\leq \frac {\varepsilon} 4$. This implies that
\[\sup_{x\in\R^{n+1}}\int_{  M_T\cap B_{R_0}(x)} |A(T)|^{4}d\mu\leq \frac 12 \varepsilon . \]

 By  a similar proof to the one in
Proposition 4.6 in \cite{KS1} (see more details in our Section 5), we can obtain
 $$\sup_{0\leq t\leq T}\left\| \nabla^{l} A(t) \right\|_{L^{\infty}} \leq c\big ( l, T, \sum_{j=0}^{m+l} \left\| \nabla^{j} A_{0} \right\|_{L^2} \big )$$
 for a large enough integer $m$.
 This shows that $A(T)$ is smooth and   $T$ is not the maximal time satisfying \eqref{M.0}.
 Therefore $T=\delta R^4$ with $R\geq \frac  {R_0}4$.
This proves Theorem  \ref{T1}.

\end{proof}
\subsection{The case of $n=5$}

In this subsection,  we complete  the proof of Theorem  \ref{T1} for  $n=5$.
Let $T$ be the maximal time of existence   satisfying \eqref{M.0} with $T=\delta R^4$.
Arguing as in the proof of \eqref{GN1}, we apply the Michael-Simon Sobolev inequality   and H\"older's inequality to obtain
\begin{align}\label{GN2}
 \left (\int_{M_t} |u|^{5} d\mu \right )^{1/5} \leq C \left (  \int_{M_t} |\nabla^2 u|^2 \,d\mu \right )^{3/8} \left
 (   \int_{M_t} |  u|^2  \,d\mu\right )^{1/8}
\end{align}
 for any $u\in
C^{1}_{c}\left( M \right)$ satisfying  that $\int_{[u\neq 0]}|H|^5\,d\mu\leq \varepsilon$ with a sufficiently small $\varepsilon$.

 Choosing $u=|A|\eta^2$ in \eqref{GN2} and noting that $|\nabla |A \eta^2||\leq |\nabla (A\eta^2)|$, we apply  \eqref{4.12}  to obtain
\begin{align}\label{5.19}
 & \left ( \int_{  M_t\cap B_{R}(x)}   |A|^{5}\  d\mu\right )^{1/5}\leq C \left ( \int_{M_t}  |\nabla^2 (A\eta^2 )|^2 \,d\mu \right )^{3/8} \left
 ( \int_{M_t}  | A|^2 \eta^4 \,d\mu \right )^{1/8} \\
  \leq& C \left ( \int_{  M_t\cap B_{2R}(x)}  |\nabla^2 A|^2 \,d\mu
 \right )^{3/8}\left ( \int_{  M_t\cap B_{2R}(x)}  | A|^2  \,d\mu\right )^{1/8}\nonumber\\
 &+ \frac C{R^{3/2}}\left ( \int_{  M_t\cap B_{2R}(x)} \left| A \right|^{2} d\mu \right )^{1/2}. \nonumber
\end{align}

Assume that $R\leq \frac {R_0}4$.
   Let $\{B_{R/2}(x_i)\}_{i=1}^L$ be an open cover of $\bar B_{R_0}(x_0)$ such that   each $B_{2R}(x)$ can be covered by a finite number $L$
   of $B_R(x_i)$ and
   $$\bigcup_{i=1}^L B_{4R}(x_i)\subset B_{2R_0}(x_0).$$
In view of \eqref{f.0} and   \eqref{f4}, we apply \eqref{5.19} to obtain
\begin{align}\label{5.20}
 &\int_{ M_T\cap  B_{R_0}(x_0)} |A(T)|^{5}d\mu\leq \sum_{i}  \int_{  M_T\cap B_{R}(x_i)} |A(T)|^{5}d\mu \\
 \leq  &C\sum_{i} \left ( \int_{  M_T\cap B_{R}(x_i)} R |\nabla^2 A(T)|^2 \,d\mu
 \right )^{15/8}\left ( \int_{ M_T\cap B_{R}(x_i))} \frac { \left| A(T) \right|^{2}} {R^3}    \,d\mu\right )^{5/8}\nonumber\\
 &+ C\sum_{i}   \left ( \int_{  M_T\cap B_{R}(x_i))} \frac { \left| A(T) \right|^{2}} {R^3} d\mu \right )^{5/2} \nonumber \\
 &\leq C\sum_{i}\left ( \frac 1{R^3}\int_{  M_0\cap B_{2R}(x_i))} \left| A_0 \right|^{2} d\mu_0 \right )^{5/2} \nonumber \\
  &+C\sum_{i} \left ( \frac C {R^4} \int_0^T \int_{  M_t\cap B_{2R}(x_i))}\frac {| A |^{2}}{R^3}
 d\mu\, dt \right )^{5/2}
\nonumber .
\end{align}
Using Lemma \ref{Vo} and H\"older's inequality, we have
\begin{align}\label{5.21}
 \left ( \frac 1{R^3}\int_{  M_0\cap B_{2R}(x)} \left| A_0 \right|^{2} d\mu_0 \right )^{5/2}\leq C \int_{  M_0\cap B_{2R}(x))} \left| A_0
 \right|^5 d\mu_0
\end{align}
and
\begin{align}\label{5.22}
  \left ( \frac C {R^4} \int_0^T \int_{  M_t\cap B_{2R}(x)}\frac {| A |^{2}}{R^3}
 d\mu\, d t \right )^{5/2} \leq     \frac C {R^4} \int_0^T \int_{  M_t\cap B_{2R}(x)}   | A |^5
 d\mu\, dt \mbox{.}
\end{align}
Combining \eqref{5.21}-\eqref{5.22} with \eqref{5.20} yields
\begin{align}\label{5.20.2}
 &\quad \int_{  M_t\cap B_{R_0}(x_0)} |A(t)|^{5}d\mu\\
 &\leq   \frac  C{\delta} \int_{  M_0\cap B_{2R_0}(x_0))}|A_0|^5 d\mu_{0}+  \frac {C} {R^4} \int_0^T\int_{  M_t\cap B_{2R_0}(x_0)}|A|^5\,d\mu dt \nonumber
\end{align}
for $t\in [0, T]$ with $T=\delta R^4$ with $R\geq \frac 14 {R_0}$. Using \eqref{5.20.2}, we  prove   Theorem  \ref{T1} for $n=5$.\qed

\section{The Gagliardo-Nirenberg  inequality and Proof of Theorem \ref{T2}}
In this section, we derive  two  local  Gagliardo-Nirenberg inequalities on higher order derivatives  to prove  Theorem  \ref{T2}.

At first, we recall the well-known Gagliardo-Nirenberg inequality in \cite {N} on $\Bbb R^n$:
For any three numbers $p, q, r \geq 1$ satisfying $\frac 1 p =\frac j n+\theta (\frac 1 r -\frac m n) +\frac {1-\theta}q$ and $\frac jm \leq
\theta
\leq 1$, we have
\[\|\nabla^j u\|_{L^p(\Bbb R^n)}\leq C||\nabla^m u\|^{\theta}_{L^r(\Bbb R^n)}\| u\|^{1-\theta}_{L^q(\Bbb R^n)},\,\, u\in L^q(\Bbb R^n)\cap W^{m,r}(\Bbb R^n).\]
If $n$ is even, we choose that $p=n$, $m=\frac n2$, $j=0$, $r=q=2$ and   $\theta= 1-\frac 2 n$. Then
\[\|u\|_{L^n(\Bbb R^n)}\leq C||\nabla^{\frac n 2} u\|^{1-\frac 2 n}_{L^2(\Bbb R^n)}|| u\|^{\frac 2 n}_{L^2(\Bbb R^n)}.\]
If $n$ is odd, we choose that $p=n$, $m=\frac {n-1}2$, $j=0$, $r=q=2$ and   $\theta=\frac {n-2}{n-1}$. Then
\[\|u\|_{L^n(\Bbb R^n)}\leq C||\nabla^{\frac {n-1} 2} u\|^{\frac {n-2}{n-1} }_{L^2(\Bbb R^n)} \| u\|^{\frac 1{n-1} }_{L^2(\Bbb R^n)}.\]

Next, we   derive two special Gagliardo-Nirenberg  inequalities on a hypersurface $M$. Let $p$ be a number with $1\leq p<n$.
According to  the Michael-Simon Sobolev inequality \eqref{MS2} and using H\"older's inequality,   we have
\begin{align}\label{MS4}\left( \int_{M} |v|^{\frac {np}{n-p}} d\mu \right)^{\frac{n-p}{n}}
&\leq C \int_{M} \left| \nabla v\right|^p d\mu +C
\left (\int_{[v\neq 0]}   \left| H \right|^n
  d\mu\right )^{\frac p n}   \left( \int_{M} \left| v\right|^{\frac {np}{n-p}} d\mu \right)^{\frac {n-p}n}
\end{align}
for any $v\in
C^{\infty}_{0}\left( M \right)$, where $C$ is  a constant
independent
of $v$. Then,
for each $v\in
C^{1}_{c}\left( M \right)$  satisfying that $\int_{[v\neq 0]}|H|^n d\mu\leq \varepsilon$ with sufficiently small $\varepsilon$, we have
\begin{align}\label{MS5}\left( \int_{M} |v|^{\frac {np}{n-p}} d\mu \right)^{\frac{n-p}{np}}
\leq C \left( \int_{M} \left| \nabla v\right|^p d\mu \right)^{\frac 1p }.
\end{align}
Thus, all  standard Sobolev    inequalities  hold  for each $v\in
C_0^{\infty}(M)$  satisfying that $\int_{[v\neq 0]}|H|^n\leq \varepsilon$ with sufficiently small $\varepsilon$.

We  prove two special Gagliardo-Nirenberg  inequalities on $M$ in the following:
\begin{lemma}\label{LGN}Let $v\in
C_0^{\infty}(M)$ satisfy  that $\int_{[v\neq 0]}|H|^n\,d\mu\leq \varepsilon$ for a sufficiently small $\varepsilon$.
Then we have
\begin{align}\label{GNe}
\|v\|_{L^n(M)}\leq C\| v\|^{1-\frac {n-2}{2[\frac n2]} }_{L^2(M)}\, \|\nabla^{[\frac n 2]} v\|^{\frac {n-2}{2[\frac n2]}}_{L^2(M)},
\end{align}
where $[\frac n 2]$ is the integer part of $\frac n 2$.
\end{lemma}
\begin{proof}Without loss of generality, we assume that $n$ is even.
It follows from using \eqref{MS5} that
\begin{align}\label{GNe1}
\int_M  |v|^{n} d\mu  \leq C\left ( \int_M |\nabla v|^{n/2}d\mu  \right )^2 .
\end{align}
Integrating by parts  yields
\begin{align*}
\int_M |\nabla v|^{n/2}d\mu  \leq &C\int_M |\nabla v|^{\frac n 2 -2}|v|\,|\nabla^2  v|  d\mu.
\end{align*}
Then using   Young's inequality, we get
\begin{align}\label{GNe3}
\int_M |\nabla v|^{n/2}d\mu  \leq C\int_M  |v|^{\frac n 4}\,|\nabla^2  v|^{\frac n 4}  d\mu .
\end{align}

When   $n$ is an even integer $n\geq 6$, we  apply  H\"older's inequality and \eqref{GNe3} to obtain
\begin{align}\label{GNe4}
\left (\int_M |\nabla v|^{n/2}d\mu \right )^2 \leq C\left (\int_M  |v|^{\frac n 4}\,|\nabla^2  v|^{\frac n 4}  d\mu \right )^2\leq C \int_M
|v|^{\frac n 2} d\mu \int_M  |\nabla^2  v|^{\frac n 2}  d\mu  .
\end{align}
By the standard interpolation inequality, we obtain
\begin{align}\label{GNe5}
 \int_M  |v|^{\frac n 2} d\mu \leq C\left (\int_M  |v|^{2} d\mu\right )^{\frac n{2(n-2)} }\left (\int_M  |v|^{n} d\mu\right )^{\frac
 {n-4}{2(n-2)}}.
\end{align}
Plugging \eqref{GNe4}-\eqref{GNe5} in  \eqref{GNe1} yields
\begin{align}\label{GNe6}
\int_M  |v|^{n} d\mu   \leq C   \int_M  |v|^{2} d\mu \left ( \int_M |\nabla^{2} v|^{\frac n2} d\mu  \right )^{\frac {2n-4} n}.
\end{align}
For an even integer $n\geq 6$,   we apply  \eqref{MS5} and the standard proof in Corollary 7.11 of \cite{GT} to obtain the higher
order Sobolev inequality:
\begin{align}\label{GNe4.1}
\left (\int_M  |\nabla^2  v|^{\frac n 2}  d\mu\right )^{2/n}\leq C\left (\int_M  |\nabla^{\frac n 2}  v|^2  d\mu\right )^{1/2}.
\end{align}
Combining \eqref{GNe4.1} with \eqref{GNe6}, we obtain \eqref{GNe} for an even $n\geq 6$.
Similarly, we can obtain \eqref{GNe} for an odd $n\geq 5$.
 \end{proof}
Two  Gagliardo-Nirenberg  inequalities in Lemma \ref{LGN} are global versions, which is not enough for proving Theorem  \ref{T2}.  Due
to the
higher derivatives of a cut-off function $\eta = \tilde
\eta \circ f$ on $M$,   it causes a difficulty  to get local Gagliardo-Nirenberg  inequalities. To overcome this difficulty, we avoid the
higher derivatives of the cut-off function to derive two local
Gagliardo-Nirenberg  inequalities on $M$:
\begin{lemma}Assume  that $\int_{M \cap B_{2R}(x)}|A|^n\,d\mu\leq \varepsilon$ for a sufficiently small $\varepsilon$.
Then   we have
\begin{align}\label{Mu1.1}
\|A\|_{L^{n}(f^{-1} (M \cap B_{R/2}(x))}\leq C\| A\|^{1-\frac {n-2}{2[\frac n2]} }_{L^2(M \cap B_{2R}(x))}\sum_{k=0}^{[\frac n2]}\left \|\frac {\nabla^{k}
A}{R^{\frac
n2 -k}}\right \|^{ \frac {n-2}{2[\frac n2]}}_{L^2( M \cap B_{2R}(x) )}.
\end{align}
\end{lemma}
\begin{proof}Without loss of generality, we assume that $n$ is even.
 For two $s, t$ with $0<\rho<s\leq 2R$,   let
$\eta_1 = \tilde \eta_1 \circ f$ be a cut-off function, where $\tilde
\eta_1$ is a $C^{2}$ function with compact support in $B_{s}(x)$  of $\mathbb{R}^{n+1}$,   $\tilde \eta_1 (x)=1$ for  $x\in
B_{\rho}(x)$, satisfying
$$\left| \nabla \eta_1\right| \leq \frac {\Lambda}{s-\rho} \mbox{ and } \left| \nabla^{2} \eta_1 \right| \leq \Lambda \left(\frac 1
{(s-\rho)^2}+\frac {
\left| A \right|}{s-\rho}
\right )
 $$
for a constant $\Lambda$ depending on the $C^{2}$ norm of $\tilde \eta_1$.
Choosing $v=\eta_1 u$ in \eqref{MS5}, we have
\begin{align}\label{LS}\left( \int_{ M \cap B_{\rho}(x)} |u|^{p^*} d\mu \right)^{\frac 1{p^*}}
\leq C \left( \int_{ M \cap B_{
s}(x)} \left| \nabla u\right|^p +\frac {|u|^p}{(s-\rho)^p}d\mu \right)^{\frac 1p }
\end{align}  with $p^*=\frac {np}{n-p}$ and $p\geq 1$.
 Using \eqref{LS} with $u=|\nabla^{\frac n 2-1}  A|$, we have
\begin{align}\label{GNe4.1.1}
\left (\int_{ M \cap B_{\frac 32 R}(x)}  |\nabla^{\frac n 2-1}  A|^{2^*}  d\mu\right )^{1/2^*}
\leq C\left ( \int_{ M \cap B_{2R}(x)} (\frac { |\nabla^{\frac n 2 -1}  A|^2 }{R^2}+    |\nabla^{\frac n 2}  A|^2)   d\mu\right )^{1/2}.
\end{align}
 Combining \eqref{LS} with \eqref{GNe4.1.1} yields
\begin{align}\label{GNe4.2}
&\quad\left (\int_{M \cap B_{\frac 54 R}(x)}  |\nabla^{\frac n 2-2}  A|^{{2^*}^*}  d\mu\right )^{1/{2^*}^*}\\
&\leq C\left (\int_{ M \cap B_{\frac 32 R}(x)} \left ( |\nabla^{\frac n 2-1}  A|^{2^*}+\frac 1{ R^{2^*}}|\nabla^{\frac n 2-2}  A|^{2^*}
\right )   d\mu\right )^{1/2^*} \nonumber\\
& \leq C\left (\sum_{k= \frac n2-2}^{\frac n 2} \int_{  M \cap B_{2R}(x)}
\frac { |\nabla^{k}  A|^2 }{R^{ n-2k}} d\mu\right )^{1/2}.\nonumber
\end{align}
Repeating  the above  procedure yields
\begin{align}\label{GNe4.3}
\left (\int_{M \cap B_{R}(x)}  |\nabla^2  A|^{\frac n 2}  d\mu\right )^{\frac 2n}\leq C\left (\sum_{k=2}^{\frac n 2} \int_{M \cap B_{2R}(x)}
\frac { |\nabla^{k}  A|^2 }{R^{ n-2k}} d\mu\right )^{1/2}.
\end{align}
Similarly, we apply \eqref{LS} to obtain
\begin{align}\label{GNe4.4}
\left (\int_{ M \cap B_{R}(x)}  |\nabla   A|^{\frac n 2}  d\mu\right )^{2/n}\leq C\left (\sum_{k=1}^{\frac n 2-1} \int_{ M \cap B_{2R}(x)}
\frac { |\nabla^{k}  A|^2 }{R^{ n-2-2k}} d\mu\right )^{1/2},
\end{align}
\begin{align}\label{GNe4.5}
\left (\int_{M \cap B_{R}(x)}  |A|^{\frac n 2}  d\mu\right )^{2/n}\leq C\left (\sum_{k=0}^{\frac n 2-2} \int_{ M \cap B_{2R}(x)}
\frac { |\nabla^{k}  A|^2 }{R^{ n-4-2k}} d\mu\right )^{1/2}.
\end{align}
Using \eqref{MS5} again, we note
\begin{align}\label{GNe4.6}
\left (\int_{ M \cap B_{R/2}(x)}  |A|^{n }  d\mu\right )^{1/n}\leq C\left ( \int_{M \cap B_{R}(x)}\left (
 |\nabla   A|^{\frac n 2} +\frac {|A|^{\frac n 2}}{R^{\frac n 2}}\right ) d\mu\right )^{\frac 2 n}.
\end{align}
Choose $v=|A|\eta$ in  \eqref{GNe6} for a cut-off function $\eta$ in \eqref{Cut} such that
\[|\nabla^2(\eta A)|\leq C|\nabla^2A| +\frac C R|\nabla A |+C|A|\left (\frac 1 {R^2} +\frac {|A|}R\right ).\]
Applying \eqref{GNe4.3}-\eqref{GNe4.6} to  \eqref{GNe6}, we obtain \eqref{Mu1.1}.
\end{proof}

As an application, we have

\begin{lemma}Let $A\in
C^{\infty}(M)$ satisfy  that $\int_{M \cap B_{2R}(x)}|A|^n\,d\mu\leq \varepsilon$ for a sufficiently small $\varepsilon$.
Then we have
\begin{align}\label{Ife}
\|A\|_{L^{\infty}(M \cap B_{R}(x)}\leq C\| A\|^{\frac {2m-n}  {2m} }_{L^2( M \cap B_{2R}(x)}\sum_{k=0}^{ m}
\left \|\frac {\nabla^k A}{R^{m -k}}\right \|^{\frac n
{2m}}_{L^2( M \cap B_{2R}(x)) }
\end{align}
for the integer $m=[\frac n2 ]+1> \frac n 2$.
\end{lemma}
\begin{proof} Without loss of generality, we assume that $n$ is even.
 For each $p>n$ and any $v\in
C_0^{\infty}(M)$, a similar proof  to  one of Theorem 5.6 in \cite{KS1}  (also \cite{LSU}) yields
\begin{align}\label{M-L2.2.}
\|v\|_{L^{\infty}(M)}\leq C \|v\|_{L^2(M)}^{1-\theta}\left (\|\nabla v\|_{L^p(M)}+\|H v\|_{L^p(M)}\right )^{\theta},
\end{align}
where $\frac 1{\theta}=2\left (\frac 1 n-\frac 1 p \right )+1$.
If $v\in
C_0^{\infty}(M)$ satisfy  that $\int_{[v\neq 0]}|H|^n\,d\mu\leq \varepsilon$ for a sufficiently small $\varepsilon$, then
\eqref{M-L2.2.} can be improved to
\begin{align}\label{Mu2.2}
\|v\|_{L^{\infty}(M)}\leq C \|v\|_{L^2(M)}^{1-\theta}\|\nabla v\|_{L^p(M)}^{\theta},
\end{align}
where $\frac 1{\theta}=2\left (\frac 1 n-\frac 1 p \right )+1$.

In fact, we can  prove \eqref{Mu2.2} by modifying the proof of Theorem 5.6 in \cite{KS1}  in  the following:

Let $C_p\|\nabla v\|_{L^p(M)}=1$ and $q=\frac p{p-1}<\frac n{n-1}$. For any $\tau\geq 0$, it follows from the improved Michael-Simon Sobolev inequality \eqref{MS5} that
\begin{align*}
&\|v^{1+\tau}\|_{L^{ \frac n{n-1}}(M)}\leq  C  \|\nabla v^{1+\tau}\|_{L^1(M)} \\
&\leq (1+\tau )    \| v^{\tau}\|_{L^q(M)} C_p\|\nabla v\|_{L^p(M)}  = ( 1+\tau )    \| v^{\tau}\|_{L^q(M)}.  \end{align*}
Choosing $k= \frac {\frac n{n-1}} q\in (1, \frac n{n-1})$, we have
 \begin{align}\label{RL3.1}
\|v\|_{L^{k(1+\tau )q }(M)}
&\leq   ( 1+\tau )^{\frac 1 {1+\tau} }   \| v \|_{L^{\tau q}(M)}^{\frac {\tau }{1+\tau }}.  \end{align}
Using \eqref{RL3.1}, we repeat the same proofs of Theorem 5.6 in \cite{KS1}  to obtain \eqref{Mu2.2}.

Choose that $p=2n$ in  \eqref{Mu2.2} with $\theta=\frac n{n+1}$, which satisfies
$
\frac 1 n=\frac 2 p=\frac 1 2 -\frac {\frac n 2 -1} n$.
By a similar proof of \eqref{GNe3}, we have
\begin{align}\label{Hs}
&\left (\int_M  |\nabla   v|^{2n}  d\mu\right )^{\frac 1{2n}} \leq C\left ( \int_M  |v|^{n}  |\nabla^2 v|^{n}  d\mu \right )^{\frac 1{2n}}\\
&\leq C \|v\|^{\frac 12} _{L^{\infty}(M)}\left (  \int_M    |\nabla^2 v|^{n}  d\mu \right )^{\frac 1{2n}}\nonumber .
\end{align}
 Combining \eqref{Hs} with \eqref{Mu2.2}, we have
\begin{align}\label{Mu2.3}
 \|v\|_{L^{\infty}(M)}\leq C \| v\|^{\frac 2 {n+2} }_{L^2(M)}\|\nabla^{2} v\|^{\frac n {n+2}}_{L^n(M)}.
\end{align}
Then,   chooseing $v=|A| \eta $ in \eqref {Mu2.2} and \eqref{Mu2.3} for a cut-off function $ \eta $ in $B_{\frac 32 R}$ and
noting the fact that $|\nabla |A \eta||\leq |\nabla (A\eta)|$ in \eqref{Mu2.2}, we have
\begin{align}\label{Mu2.3.12}
 \|A  \eta\|_{L^{\infty}(M)}\leq C \| A \eta\|^{\frac 2 {n+2} }_{L^2(M)}\|\nabla^{2}  (A\eta)\|^{\frac n {n+2}}_{L^n(M)}.
\end{align}
Note that
\begin{align}\label{Mu2.3.1}
|\nabla^2(\eta A)|\leq C|\nabla^2A| +\frac C R|\nabla A |+C\left (\frac {|A|}{R^2} +\frac {|A|^2}R\right ).
\end{align}

 Repeating   a similar proof  to prove \eqref{GNe4.3},  we obtain
\begin{align}\label{BL1}
\left (\int_{ M \cap B_{R}(x)}  |\nabla^3  A|^{\frac n 2}  d\mu\right )^{\frac 2n}\leq C\left (\sum_{k=3}^{\frac n 2 +1} \int_{M \cap B_{2R}(x))}
\frac { |\nabla^{k}  A|^2 }{R^{ n-2k}} d\mu\right )^{1/2}.
\end{align}

Using the Sobolev inequality \eqref{MS5} with $p=\frac n 2$,  \eqref{GNe4.3}  and \eqref{BL1}, we have
\begin{align}\label{BL3}
&\left (\int_{M \cap B_{R/2}(x)}   |\nabla^2 A|^{n} d\mu \right )^{1/n}
\leq C\left ( \int_{M \cap B_{R}(x)}\left (|\nabla^3   A|^{\frac n 2} +\frac {|\nabla^2 A|^{\frac n 2}}{R^{\frac n 2}}\right ) d\mu\right )^{\frac 2 n}\\
 &\leq C\left (\sum_{k=0}^{\frac n 2 +1} \int_{M \cap B_{2R}(x)}
\frac { |\nabla^{k}  A|^2 }{R^{ n-2k}} d\mu\right )^{1/2}.\nonumber
\end{align}
Similarly,   we apply \eqref{MS5} to obtain
\begin{align}\label{GNe4.6}
\left (\int_{M\cap  B_{R/2}(x)}  |\nabla A|^{n }  d\mu\right )^{1/n}\leq C\left ( \int_{ M\cap B_{R}(x)}\left (
 |\nabla^2   A|^{\frac n 2} +\frac {|\nabla A|^{\frac n 2}}{R^n }\right ) d\mu\right )^{\frac 2 n},
\end{align}
\begin{align}\label{GNe4.6}
&\left (\int_{ M\cap B_{R/2}(x)}  |A|^{2n }  d\mu\right )^{1/n}\leq C\left ( \int_{M\cap  B_{R}(x)}\left (
 |\nabla   |A|^2|^{n/2} +\frac {|A|^{n}}{R^{n/2} }\right ) d\mu\right )^{\frac 2 n}\\
 &\leq C \left ( \int_{M\cap  B_{R}(x)}  |A|^{n}   d\mu\right )^{\frac 1 n}\left ( \int_{M\cap B_{R}(x)}\left (
 |\nabla   A|^{n} +\frac {|A|^{n}}{R^{n} }\right ) d\mu\right )^{\frac 1 n}\nonumber .
\end{align}
Applying \eqref{Mu2.3.1}-\eqref{GNe4.6}   to \eqref{Mu2.3.12},  we have \eqref{Ife}.
\end{proof}

 \begin{lemma}\label{T:c2}Let $f(\cdot,t)$ be the smooth solution to  \eqref{E:theflow} and $\eta$ be a cut-off function.
  For any $s\geq 2m+4$, we have
 \begin{align*}
   &\frac{d}{dt} \int_{M_t}  \left| \nabla^{m} A \right|^{2} \eta^{s} d\mu + \frac{1}{2} \int_{M_t}  \left| \nabla^{m+2} A \right|^{2}  \eta^{s}
   d\mu \\
 \leq & C  \left\| A \right\|^{4}_{\infty, \left[  \eta >0\right]}     \int_{M_t} \left| \nabla^{m} A \right|^{2}
  \eta^{s} d\mu +C \left(  \frac 1{R^{4}} + \left\| A \right\|^{4}_{\infty, \left[  \eta >0\right]}  \right)\frac 1{R^{2m}}\left\| A
  \right\|^{2}_{2, \left[  \eta>0 \right]}  \mbox{.}
\end{align*}
 \end{lemma}
\begin{proof} In fact, the proof for $n>2$ is the same
proof of  Proposition 4.5   in \cite{KS1}. We would like to outline some main steps of proofs with $R=1$.

The  first step in \cite{KS1} is  the following two
interpolation inequalities:

(i) Let $2\leq p < \infty$, $k
\in \mathbb{N}$ and $s\geq kp$. Then for any tensor $\phi$ on $M=M_t$, we have
  \begin{align} \label{E:int0}
  &\left( \int_{M} \left| \nabla^{k} \phi \right|^{p} \eta^{s} d\mu \right)^{\frac{1}{p}} \leq \varepsilon \left( \int_{M} \left| \nabla^{k+1} \phi \right|^{p} \eta^{s+p} d\mu \right)^{\frac{1}{p}} +C \left( \int_{\left[
 \eta >0 \right]} \left| \phi \right|^{p} \eta^{s-kp} d\mu \right)^{\frac{1}{p}} \mbox{,}
\end{align}
where $C=C\left( \varepsilon, k, s, p, \Lambda \right)$.

 (ii)  Let $0 \leq i_{1}, \ldots, i_{r} \leq k$, $i_{1} + \cdots + i_{r} =2k$ and $s\geq 2k$.  For a tensor  $\phi$ on $M$, we have
 \begin{align} \label{5.28}
  \int_{M} \left| \nabla^{i_{1}} \phi \star \cdots \star \nabla^{i_{r}} \phi \right| \eta^{s} d\mu
    \leq C \left\| \phi \right\|^{r-2}_{\infty, \left[  \eta>0\right]} \left( \int_{M} \left| \nabla^{k} \phi \right|^{2} \eta^{s} d\mu +
    \left\| \phi \right\|^{2}_{2, \left[ \eta >0\right]} \right)
\end{align}
for a constant $C=C\left( k, s, r, \Lambda\right)$.

Then it follows from Lemma \ref{T:e2}   that
\begin{align*}
  & \quad \frac{d}{dt} \int_{M}  \left| \nabla^{m} A \right|^{2} \gamma^{s} d\mu + \frac{3}{4} \int_{M}  \left| \nabla^{m+2} A \right|^{2}
  \gamma^{s} d\mu \\
&  \leq  \int_{M} \left[ P_{3}^{m+2}(A)+ P_{5}^{m}(A) \right] \star \nabla^{m}A \gamma^{s} d\mu  + C\int_{\left[ \gamma>0\right]} \left| A
\right|^{2} \gamma^{s-4-2m} d\mu \mbox{.}\nonumber
\end{align*}
  Using  \eqref{5.28},  we have
$$  \int_{M} P_{5}^{m}\left( A \right) \star \nabla^{m} A  \eta^{s} d\mu \leq C \left\| A \right\|_{\infty, \left[  \eta
>0\right]}^{4} \left( \int_{M} \left| \nabla^{m} A \right|^{2}  \eta^{s} d\mu + \left\| A \right\|_{2, \left[ \eta >0\right]}^{2}
\right).$$
Combining \eqref{E:int0} with \eqref{5.28} yields
\begin{align*}
  & \int_{M}  P_{3}^{m+2}\left( A \right) \star \nabla^{m} A  \eta^{s} d\mu
 \leq  C \left\| A \right\|_{\infty, \left[  \eta >0\right]}^{2}
 \left( \int_{M} \left| \nabla^{m+1} A \right|^{2}  \eta^{s} d\mu + \left\| A \right\|_{2, \left[  \eta >0\right]} \right) \\
   \leq &\frac 1 4 \int_{M} \left| \nabla^{m+2} A
\right|^{2}  \eta^{s} d\mu + C
 \left\| A \right\|_{\infty, \left[   \eta >0\right]}^{4}
 \int_{M} \left| \nabla^{m} A \right|^{2}  \eta^{s} d\mu   +  C \left\| A \right\|_{\infty, \left[  \eta >0\right]}^{2} \left\| A \right\|_{2, \left[  \eta >0\right]}
\mbox{.}
   \end{align*}
 Using the above estimates, we complete  the
  proof.  \end{proof}
\begin{proof}[Proof of Theorem   \ref{T2} ]
Let $f(\cdot, t)$ be the smooth solution to the biharmonic flow \eqref{E:theflow} in $[0,  T]$, where $T$ is the maximal time such that
\begin{align} \label{M.0.1}\sup_{0 \leq t \leq T,\,
x\in\R^{n+1}} \int_{ M_t\cap B_{R_0}(x))}   |A|^{n}\  d\mu \leq \varepsilon
  \end{align}
for  a sufficiently small $\varepsilon$ and
\begin{align} \label{M.0.2}  T  \sup_{0\leq t\leq T}\|A(t)\|^4_{L^{\infty}(M)}   \leq C_1
  \end{align} for a  constant $C_1>0$.
Then, we choose  two constants $\delta$ and $R$ such that $T= {\delta}{R^4}$ for  a sufficiently small $\delta$.
We may assume that $R<\frac {R_0}4$. Otherwise, the proof is complete.

 Let  $\tilde \eta$ be a $C^{\infty}$ cut-off function with compact support in $B_{2R}(x_0)$  of
 $\mathbb{R}^{n+1}$, where $\tilde \eta (x) =1$ for  $x\in
B_{R}(x_0)=1$.  We   set $\eta = \tilde \eta \circ f$. Then
$$\left| \nabla \eta\right| \leq \frac {C}{R} \mbox{ and } \left| \nabla^{2} \eta \right| \leq C \left(\frac 1 {R^2}+\frac {
\left| A \right|}{R}\right)
\mbox{,}$$
where $C$ is a bound, depending on the $C^{2}$-norm of $\tilde \eta$.

   Let $\{B_{R/2}(x_i)\}_{i=1}^{\infty}$ be an open cover of $\R^{n+1}$, where   each $B_{2R}(x)$ can be covered by a finite number $L$
   of $B_R(x_i)$.  For each $x_0\in\R^{n+1}$, there exists a finite number $L_{x_0}$ such that
   $$B_{R_0}(x_0)\subset\bigcup_{i=1}^L B_{R/2}({x_0}_i)\subset\bigcup_{i=1}^L B_{4R}({x_0}_i)\subset B_{2R_0}(x_0).$$
  It follows from  \eqref{f.0} that for  each $R$ with $4R\leq R_0$, we obtain
  \begin{align} \label{Cover1}
  &\sup_{x_0\in\R^{n+1}, 0\leq t\leq T}\sum_{i=1}^{L_{x_0}}\left (\int_{M_t\cap  B_R({x_0}_i)}\frac  {\left| A(t)\right|^{2}} {R^{n-2}} d\mu\right )^{n/2}  \\
 & \leq \sup_{x_0\in\R^{n+1}}\sum_{i=1}^{L_{x_0}}\left (\int_{M_0\cap  B_{2R}({x_0}_i)}\frac {\left| A_{0} \right|^{2}}{R^{n-2}}d\mu_{0}\right )^{n/2}  \nonumber\\
 &\quad +C \left (\frac {T} {R^4}\right )^{n/2} \sup_{x_0\in\R^{n+1}, 0\leq t\leq T}\sum_{i=1}^{L_{{x_0}_i}} \left ( \int_{M_t\cap B_{2R}({x_0}_i)} \frac {\left| A\right|^{2}}{R^{n-2}}
 d\mu \right )^{n/2}  \mbox{.}\nonumber
   \end{align}
Using $T=\delta R^4$ with  sufficiently small $\delta$, we have
\begin{align} \label{Cover2}
 &\sup_{x_0\in\R^{n+1}, 0\leq t\leq T}\sum_{i=1}^{L_{x_0}}\left (\int_{M_t\cap B_R({x_0}_i)}\frac  {\left| A(t)\right|^{2}} {R^{n-2}} d\mu\right )^{n/2}\\
  &\leq C\sup_{x_0\in\R^{n+1}}\sum_{i=1}^{L_{x_0}}\left (\int_{M_0\cap B_{2R}({x_0}_i)}\frac {\left| A_{0} \right|^{2}}{R^{n-2}}d\mu_{0}\right )^{n/2}\nonumber\\
& \leq C\sup_{x_0\in\R^{n+1}}\sum_{i=1}^{L_{x_0}}\int_{M_0\cap B_{2R}({x_0}_i)} \left| A_{0} \right|^{n} d\mu_{0} \nonumber
\leq C\sup_{x_0\in\R^{n+1}} \int_{M_0\cap  B_{R_0}({x_0})} \left| A_{0} \right|^{n} d\mu_{0} .
   \end{align}

Using \eqref{f.0.2} and \eqref{5.1.4}, there exists $t_1=t_2\in [ \frac T4 , \frac T2]$ such that
\begin{align} \label{Cover3}
   &  \sup_{t_2\leq t\leq T}\int_{M_t\cap B_{R}(x) } R^{6-n} \left| \nabla^{2} A(t)  \right|^{2}+  R^{4-n}|\nabla A(t)|^2  d\mu\\
 & \leq \frac C{\delta}\int_{M_0\cap B_{4R}(x)} \frac {\left| A_0 \right|^{2}}{R^{n-2}} d\mu_0  + \frac C {\delta R^4} \int_0^T \int_{M_t\cap B_{4R}(x))}
  \frac {\left| A\right|^{2}}{R^{n-2}}
 d\mu d t \nonumber
  \mbox{.}
\end{align}

Applying Gronwall's inequality to Lemma \ref{T:c2}, we have for $t\in [\tilde t_0,T]$ and  $m\geq 2$
\begin{align}\label{Cover4}
   & \int_{M_t}  \left| \nabla^{m} A(t) \right|^{2} \eta^{s} d\mu +\frac{1}{4} \int_{\tilde t_0}^t \int_{M_t}  \left| \nabla^{m+2} A \right|^{2}  \eta^{s}
   d\mu  d\tau  \\
 \leq & e^{C \int_{\tilde t_0}^t  \left\| A(\tau) \right\|^{4}_{\infty, \left[  \eta >0\right]}  d\tau } \left(  \int_{M_{\tilde t_0}} \left|
 \nabla^{m} A(\tilde t_0) \right|^{2}
  \eta^{s} d\mu +\frac 1{R^{2m+4}}\int_0^t \int_{\left[  \eta>0 \right]}  {|A|^2} \,d\mu d\tau\right) \mbox{.}\nonumber
\end{align}
By induction, there exists $t_m\in [ \frac T2-\frac T{2^m}, \frac T2-\frac T{2^{m+1}}]$  for each $m\geq 2$ such that
\begin{align}\label{Cover5}
   & \sup_{t_m\leq t\leq T}\int_{M_{t}\cap B_{R}(x)}  R^{2m+2-n} \left| \nabla^{m} A(t) \right|^{2}  d\mu +\frac{ R^{2m+2-n}}{4} \int_{t_m}^T \int_{ M_t\cap B_{R}(x)}  \left| \nabla^{m+2} A \right|^{2}
   d\mu  dt  \\
 \leq & C\left(  \int_{M_{t_m}\cap B_{2R}(x)}R^{2m+2-n} \left|
 \nabla^{m} A(t_m) \right|^{2} \eta^{s} (t_m) d\mu +\frac 1{R^4}\int_{t_m}^T \int_{M_t\cap B_{2R}(x)} \frac {|A|^2}{R^{n-2}} \,d\mu dt\right)\nonumber\\
  &\leq   C_m\left(  \int_{M_0\cap B_{4R}(x)} \frac {\left| A_0 \right|^{2}}{R^{n-2}} d\mu_0 +\frac 1{R^4}\int_{0}^T \int_{M_t\cap B_{4R}(x)} \frac {|A|^2}{R^{n-2}} \,d\mu dt\right) \mbox{.}\nonumber
\end{align}
Using \eqref{Cover2} and \eqref{Cover5},   we have
\begin{align}\label {5E2}
 &  \sup_{x_0\in\R^{n+1}} \sum_{i=1}^{L_{x_0}}\left ( \int_{M_T\cap B_{R}(x_{0,i})}  \frac { \left|{\nabla^{m} A(T)} \right|^{2}}  { R^{n-2m-2}} d\mu\right )^{n/2}
  \\
&\leq C_m  \sup_{x_0\in\R^{n+1}} \int_{M_0\cap B_{2R_0}({x_0})} \left| A_{0} \right|^{n} d\mu_{0}\leq C_m \varepsilon_0. \nonumber
\end{align}

In view of \eqref{Cover2} and \eqref{5E2}, we apply the Gagliardo-Nirenberg  inequality \eqref{Mu1.1} to obtain
\begin{align}\label {Cover6}
 &\int_{ M_T\cap B_{R_0}(x_0)} |A(T)|^{n}d\mu\leq \sum_{i=1}^{L_{x_0}}  \int_{M_T\cap  B_{\frac R 2}({x_0}_i)} |A(T)|^{n}d\mu \\
 \leq  &C \sum_{i=1}^{L_{x_0}}  \| A (T)\|^{n(1-\frac {n-2}{2[\frac n2]}) }_{L^2(M_T\cap B_{R}({x_0}_i))}\sum_{m=0}^{[\frac n2]}\left \|\frac {\nabla^{m}
A(T)}{R^{[\frac n2] -m}}\right \|^{ \frac {n(n-2)}{2[\frac n2]}}_{L^2(M_T\cap B_{R}({x_0}_i) )}
\nonumber \\
 \leq  &C \sum_{i=1}^{L_{x_0}} \left (\left  \|\frac { A(T)}{R^{\frac {n-2}2}}\right  \|^{n }_{L^2(M_T\cap B_{2R}({x_0}_i))}+\sum_{m=0}^{[\frac n2]}\left \|\frac {\nabla^{m}
A(T)}{R^{\frac {n-2}2 -m}}\right \|^{n}_{L^2(M_T\cap B_{2R}({x_0}_i) )}\right )
\nonumber \\
&\leq C_n  \sup_{x_0\in\R^{n+1}} \int_{M_0\cap B_{2 R_0}({x_0})} \left| A_{0} \right|^{n} d\mu_{0}\leq L C_n \varepsilon_0\leq \frac 1 2 \varepsilon. \nonumber
\end{align}
 Similarly,  noting that $l=[\frac n2 ]+1> \frac n 2$ and $T=\delta R^4$, we obtain
  \begin{align}\label{Cover7}
T^{\frac 1 4}\|A (T)\|_{L^{\infty}(M_T\cap B_{R}(x))}&\leq CR\| A(T)\|^{\frac {2l-n}  {2l} }_{L^2(M_T\cap B_{2R}(x))}\sum_{m=0}^{l}
\left \|\frac {\nabla^m A(T)}{R^{l -m}}\right \|^{\frac n
{2l}}_{L^2(M_T\cap B_{2R}(x))}\\
&\leq C_l   \varepsilon_0 \leq \frac 1{2} C_1^{1/4}  \nonumber
\end{align}
for a sufficiently small $\varepsilon_0$.

 By  a similar proof to \eqref{Cover7}, we  obtain
  \begin{align}\label{Cv1}
  \left\| \nabla^{m} A(T) \right\|_{L^{\infty}(M_T\cap  B_{R}(x)} \leq C\big (m, R, \left\|   A_{0} \right\|_{L^n( M_0\cap B_{2R_0}(x) )} \big ) \mbox{.} \end{align}
It follows from \eqref{Cover6}-\eqref{Cv1} that $A(T)$ is smooth and $T$ is not the maximal time satisfying \eqref{M.0.1}-\eqref{M.0.2}.  Therefore $T\geq \frac {\delta R_0^4}4$. This proves Theorem  \ref{T2}.
 \end{proof}

\section{The   Willmore flow in higher dimensions}
In this section, we investigate the Willmore flow in higher dimensions (i.e. $n\geq 2$).

 At first, we have
\begin{lemma}
For  each $\varepsilon>0$, let $f: M^n\times(-\varepsilon, \varepsilon)\rightarrow \mathbb R^{n+1}$ be a smooth family of compactly
supported immersion of a hypersurface $M^n$ evolving with normal velocity $\delta f = F\nu$, where $\delta=(d/d \varepsilon)|_{ \varepsilon=0}$ is the variational derivative operator. Then
\begin{align*}
\delta g_{ij} &= -2F A_{ij},\quad
\delta g^{ij}  = 2F A^{ij},\quad
\delta \nu  =-\nabla F \mbox{,}\\
\delta d \mu &= -HF~d\mu ,\quad
\delta A_{ij} = \nabla_{ij}F-F A_{jk}A_{il}g^{kl} \mbox{.}
\end{align*}
\end{lemma}

Using these evolution equations  in Lemma 6.1, the first variation formula for the Willmore energy
 functional of of a hypersurface $M$ in the normal direction is:
\begin{lemma}\label {6.2}
	The first variation  of the Willmore functional is given respectively by:
	\begin{align} \label{FV}
	\delta{\mathcal W(f)}&=\frac 12\delta\int_{M} |A|^2 d\mu  =\int_{M} F\Big\{\Delta_{M^n} H-\frac 12H|A|^2+ C(A)\Big\} d\mu,
	\end{align}
	where $C(A)=g^{ij}g^{kl}g^{mn}A_{ik}A_{lm}A_{nj}$.
\end{lemma}
\begin{proof}
	Using the above evolution equations and the divergence
	theorem, we compute
	\begin{align}
	\delta{\mathcal W(f)}
	&=\frac 12\int_{M} \Big\{2Fg^{kl}A^{ij}A_{ik}A_{jl}+2g^{ij}g^{kl}A_{jl}\nabla_{ik}F-FH|A|^2\Big\} d\mu \nonumber\\
	&=\frac 12\int_{M} \Big\{-2g^{ij}g^{kl}\nabla_iA_{jl}\nabla_{k}F +F(2g^{kl}A^{ij}A_{ik}A_{jl}-H|A|^2)\Big\} d\mu\nonumber\\
	&=\frac 12\int_{M} \Big\{-2g^{kl}\nabla_l(g^{ij}A_{ij})\nabla_{k}F +F(2C(A)-H|A|^2)\Big\} d\mu\nonumber\\
	&=\frac 12\int_{M} \Big\{-2g^{kl}\nabla_lH\nabla_{k}F +F(2C(A)-H|A|^2)\Big\} d\mu\nonumber\\
	&= \int_{M} F\Big\{\Delta_{M^n} H-\frac 12 H|A|^2+C(A)\Big\} d\mu,\nonumber
	\end{align}
	where we have used integration by part and the Codazzi equation.
	\end{proof}
Then it follows from \eqref{FV} that the Euler-Lagrange equation on the Willmore functional \eqref{W-E} is
	\begin{align}\label{EL1.1}
\Delta_{M} H-\frac{1}{2}H|A|^2+C(A)=0.
	\end{align}
	\begin{remark} When $n=2$, we may choose normal coordinates at $p\in M$ such that $g_{ij}(p)=\delta_{ij}$ and the matrix of the Weingarten map is diagonal.  Let $\lambda_1$ and $\lambda_2$ be the two principal curvatures. Then
\begin{align*}
	H= \lambda_1+\lambda_2, \quad|A|^2=\lambda_1^2+\lambda_2^2,\quad|A^0|^2=\frac{1}{2}(\lambda_1-\lambda_2)^2.
	\end{align*}	
	 Thus
	\begin{align*}
	C(A)
=(A_{11})^3+(A_{22})^3=\lambda_1^3+\lambda_2^3
=\frac{1}{2}H(|A|^2+2|A^0|^2).
	\end{align*}
	Consequently, the equation \eqref{EL1.1} for $n=2$ agrees with the 2-dimensional Willmore equation in  \cite{KS2} as follows:
	$$\Delta_{M^2} H + H|A^0|^2 =0.$$
	\end{remark}	

Now, we recall the Willmore flow for  hypersurfaces in higher dimensions
\begin{align} \label{W-flow2}
  &\frac{\partial f}{\partial t}  = -\left(\Delta_{M_t} H-\frac{1}{2}\vert A\vert^2 H+C(A)\right) \nu
   \end{align}
with initial value $f \left( \cdot, 0 \right)=f_{0}$.
Then it follows from Lemma \ref{6.2} that
\begin{lemma}
	Let  $f$ be a smooth solution of   the flow \eqref{W-flow2} in $[0,T]$. Then, for each $t\in [0, T]$, we have
	\begin{align} \label{ID}
\frac 12 \int_{M_t} |A(t)|^2 d\mu dt +\int_0^t\int_{M_{\tau}} (|\Delta_{M_{\tau}} H-\frac 12H|A|^2+ C(A)|^2) d\mu d\tau
=\frac 12 \int_{M_0} |A(0)|^2 d\mu .
	\end{align}
\end{lemma}

The same proof of Theorem \ref{T2} yields

\begin{theorem} \label{T5}
Let $f_0: M^n \rightarrow \mathbb{R}^{n+1}$ be a smooth  immersion with $n\geq 2$.
Assume that there are two absolute positive constants $\varepsilon_0 $ and $\varepsilon$  such
that
 \begin{align} \label{C6}
\int_{M_0\cap  B_{2R_0}(x) } \left| A_{0} \right|^{n} d\mu_{0} \leq \varepsilon_{0}< \varepsilon
 \end{align}
for any $x\in \mathbb{R}^{n+1}$ and some fixed  $R_0>0$.  Then the maximal
existence time $T$ of the flow \eqref{W-flow2}  satisfies
$$T \geq \frac{1}{C} R_0^4$$ for a constant $C$.
Moreover, for $0 \leq t \leq \frac{1}{C} R_0^4$, we have
$$\int_{M_t\cap B_{R_0} (x)} \left| A(t) \right|^{n} d\mu \leq C \varepsilon_0 \mbox{.}$$
\end{theorem}

Now we complete the proof of Theorem   \ref{T4}.

 \begin{proof}[Proof of Theorem   \ref{T4}]
 Let $f_0: M^n \rightarrow \mathbb{R}^{n+1}$ be a smooth  immersion with $n\geq 3$.

As an application of Theorem  \ref{T5},   there exists a global existence of a smooth solution to the Willmore flow \eqref{W-flow1} for    $t\in [0, T_1]$ with $T_1\geq CR_0^4$. Let $f(\cdot, t)$ be the smooth solution  to the Willmore flow \eqref{W-flow1} in $[0,  T]$, where $T$ is the maximal time such that
\begin{align} \label{6M.0.1}\sup_{0 \leq t \leq T,\,
x\in\R^{n+1}} \int_{M_t\cap B_{R_0}(x)}   |A(t)|^{n}\  d\mu \leq \varepsilon
  \end{align}
   and
\begin{align} \label{6M.0.1.2}  \delta R_0^4  \sup_{T-\delta R_0^4\leq t\leq T}\|A(t)\|^4_{L^{\infty}(M_t)}   \leq C_1
  \end{align} for a  constant $C_1>0$ and a sufficiently small $\varepsilon >0$.
Using the same proof in Lemma \ref{Vo}, for any $t\in [ 0, T]$, we have
\begin{align}\label{6M.0.2}
\sup_{x\in \R^{n+1}}\mu \left ( M_t\cap B_{R_0}(x)\right )\leq CR_0^n.
\end{align}
Similarly to \eqref{Cover6},   it follows from using   \eqref {ID}  that
\begin{align*}
&\quad \|A (T)\|^n_{L^{n}(M_T\cap B_{R_0}(x))}\\
 &\leq C\| A (T)\|^{n(1-\frac {n-2}{2[\frac n2]})}_{L^2(M_T\cap B_{2R_0}(x))}\sum_{m=0}^{[\frac n2]}\left \|\frac {\nabla^{m}
A(T)}{R_0^{[\frac
n2] -m}}\right \|^{\frac {n(n-2)}{2[\frac n2]} }_{L^2(M_T\cap B_{2R_0}(x) )}\\
&\leq
C\left\| \frac {A (0)}{R_0^{\frac {n-2}2}}\right \|^{n(1-\frac {n-2}{2[\frac n2]})}_{L^2(M_0)}\sum_{m=0}^{[\frac n2]}\left \|\frac {\nabla^{m}
A(T)}{R_0^{\frac {n-2}2 -m} }\right \|^{\frac {n(n-2)}{2[\frac n2]} }_{L^2(M_T\cap B_{2R_0}(x) )}\\
&\leq C_n \varepsilon_0^{\frac {n(n-2)}{2[\frac n2]}}\varepsilon^{\frac {n(n-2)}{2[\frac n2]}}\leq \frac 12 \varepsilon \end{align*}
for a sufficiently small $\varepsilon_0$.  Similarly to \eqref{Cover7},  choosing $m=[\frac n2 ]+1> \frac n 2$, we obtain
  \begin{align}\label{6Cover7.1}
& \delta^{\frac 1 4} R_0\|A (T)\|_{L^{\infty}(M_T\cap B_{R_0}(x)}\\
\leq &C_mR_0\| A(T)\|^{\frac {2m-n}  {2m} }_{L^2(M_T\cap B_{2R_0}(x)}\sum_{k=0}^{ m}
\left \|\frac {\nabla^k A(T)}{R_0^{m -k}}\right \|^{\frac n
{2m}}_{L^2(M_T\cap B_{2R_0}(x) }\nonumber\\
\leq &\frac 1{2} C_1^{1/4}  \nonumber
\end{align}
for a sufficiently small $\varepsilon_0$.

 This shows that any $T>0$ is not the maximal time satisfying \eqref{6M.0.1}-\eqref{6M.0.1.2}.
Therefore,   there exists  a global  smooth solution to the Willmore flow \eqref{W-flow1} for    $t\in [0,\infty )$. Using \eqref{6Cover7.1}, we have for each $l\geq 1$
 \begin{align}\label{6Cv2.2}
  \left\| \nabla^{l} A(t) \right\|_{L^{\infty}(M_t\cap B_{R_0}(x))} \leq C\big ( l, R_0, \varepsilon \big ) \mbox{.} \end{align}

When $n=2$,     Kuwert and  Sch\"{a}tzle in Theorem 4.2 of \cite{KS2}  proved a  localized version of the result of  Langer \cite{Lan} and completed their result as $t\to \infty$.
When $n>2$, Naff  in his appendix in  Section 4 of \cite{Na}   proved a  localized version of the result of  Langer to higher dimension.
More clearly,
applying Proposition 4.2 of  \cite{Na} (also \cite{Ha1}) to \eqref{6Cv2.2},
there exist points $x_k\in \R^{n+1}$ and diffeomorphisms $\phi_k : M^n(t_k)  \to U_k\subset M^n$ such that  as $t_k\to\infty$ (up to a subsequence),
\[f(\phi_k(p), t_k) - x_k \to f_{\infty}\]
locally in $C^k$ on $M^n$ as in \cite{KS2}. On $M^n$, we set
\[\tilde f_k(p, t)=f_k(\phi_k(p), t_k+t )-x_k,\]
which is a solution of the  Willmore flow on $\tilde M_k$.
Then as $t\to\infty$, $\tilde f_k(t)$ converges to a solution  $\tilde f_{\infty}$ of the  Willmore flow.
It follows from  Lemma \ref{ID}  that $\int_M |A (t)|^2 d\mu$ is increasing and $\lim_{t\to \infty}\int_M |A (t)|^2 d\mu$ exists.  Then
\begin{align*}
&\int_{0}^{1}\int_{\tilde M_k} |\Delta_{M_t} H_{\tilde f_k} -\frac 12H_{\tilde f_k}|A_{\tilde f_k}|^2+ C(A_{\tilde f_k})|^2   d\mu_k\,dt\\
&=
\int_{t_k}^{t_k+1}\int_{M_t} |\Delta_{M_t} H-\frac 12H|A|^2+ C(A)|^2 d\mu\,dt \to 0.\end{align*}
This shows that $f_{\infty}$ is a smooth solution of  the Willmore equation \eqref{W-EL}.
This proves Theorem  \ref{T4}.
 \end{proof}

\medskip\noindent
{\bf Acknowledgement:} {We would like to thank two referees for their very useful suggestions and comments to improve our work.


\begin{thebibliography}{20}
 \bibitem{Baker}
 C. Baker, The mean curvature flow of submanifolds of high codimension, PhD thesis, Australian National University, Canberra, 2010.

  \bibitem{BWW}
    Y. Bernard,   G. Wheeler and V.-M. Wheeler,
    \newblock{Concentration-compactness and finite-time singularities for Chen's flow},
    \newblock{J. Math. Sci. Univ. Tokyo}, {\bf 26} (2019),  no. 1, 55--139.


\bibitem{Bombieri1969}
E. Bombieri, E. De Giorgi and E. Giusti, Minimal cones and the Bernstein problem, Invent. Math. {\bf 7} (1969), 243--268.


 \bibitem{CDY}
    K. C. Chang, W. Y. Ding and R. Ye,
    \newblock{Finite-time blow-up of the heat flow of harmonic maps from surfaces},
    \newblock{J. Differential Geom.} {\bf 36} (1992),   507--515.

\bibitem{CWY}
  S. Y. A.  Chang, L. Wang   and P. Yang,
    \newblock{A regularity theory of biharmonic maps},
    \newblock{Comm. Pure Appl. Math.} {\bf 52}  (1999), 1113--1137.


 \bibitem{Chen1991} B. Y. Chen, Some open problems and conjectures on submanifolds of finite
type, Soochow J. Math. {\bf17} (1991) no. 2, 169--188.



  \bibitem{CS}
        Y. M. Chen and  M. Struwe,
        \newblock{Existence and partial regularity for heat flow for hamonic maps},
        \newblock{Math. Z.} {\bf 201} (1989), 83--103.


        \bibitem{DT}
    W. Ding and G. Tian,
    \newblock{Energy identity for a class of approximate harmonic maps from
  surfaces.},
    \newblock{ Comm. Anal. Geom.} {\bf 3} (1995), 543--554.



 \bibitem{GT}D. Gilbarg and  N. S.  Trudinger,
  {Elliptic partial differential equations of second order},  Grundlehren der Mathematischen Wissenschaften, Vol. 224, Springer-Verlag,
  Berlin-New York, 1977.


\bibitem{EL}
J. Eells and L. Lemaire, Selected Topics in Harmonic Maps. In
Proceedings of the CBMS Regional Conference Series in Mathematics,
Providence, RI, USA, 31 December 1983.

  \bibitem{ES}
        J. Eells and J.H. Sampson
        \newblock{harmonic mappings of Riemannian manifolds,}
        \newblock{Amer. J. Math.} {\bf 86} (1964), 109--160.



  \bibitem{ES1}
J. Eells and J. H. Sampson, Variational theory in fibre bundles,  Proc. U.S.-Japan Seminar
in Differential Geometry (Kyoto, 1965), Nippon Hyoronsha, Tokyo, 1966 pp. 22--33.

\bibitem{Fetcu}
D. Fetcu and C. Oniciuc, {Biharmonic and biconservative
	hypersurfaces in space forms}, Contemp. Math. {\bf 777} (2022), 65--90.

\bibitem{FHZ}
Y. Fu, M. C. Hong and X. Zhan, On Chen's biharmonic conjecture for
hypersurfaces in $\mathbb R^5$, Adv. Math. {\bf 383} (2021), 107697.

\bibitem{FHZ1}
Y. Fu, M. C. Hong and X. Zhan,  \newblock{Biharmonic conjectures on hypersurfaces in a space form,}   Trans. Amer. Math. Soc.
{\bf 376}   (2023),  8411--8445.



   \bibitem{Ha}
   R. S. Hamilton,  Three-manifolds with positive Ricci curvature, J. Diff. Geom. \textbf{17} (1982), 255--306.

  \bibitem{Ha1}
   R. S. Hamilton,   A compactness property for solutions of the Ricci flow,  {Amer. J. Math.} {\bf 117} (1995), 545--572.


\bibitem{Hasanis1995}
T. Hasanis and T. Vlachos, Hypersurfaces in $\mathbb E^4$ with
harmonic mean curvature vector field, Math. Nachr. {\bf 172} (1995),
145--169.

   \bibitem{HY1}
  M. C. Hong   and H. Yin,
    \newblock{Partial regularity of a minimizer of the relaxed energy for biharmonic maps},
    \newblock{J. Funct. Anal.} {\bf  262} (2012), 682--718.




   \bibitem{HP}
   G. Huisken and A. Poldon,  {Geometric evolution equations for hypersurfaces}, Calculus of Variations and Geometric Evolution Problems
   (Cetraro, 1996), Springer, Berlin, 1999, 45--84.

   \bibitem{jiang1987}
G. Y. Jiang, Some non-existence theorems of 2-harmonic isometric
immersions into Euclidean spaces (in Chinese), Chin. Ann. Math. Ser. A {\bf
8} (1987), 376--383.

 \bibitem{KS1}
    E. Kuwert and  R. Sch\"{a}tzle,  {Gradient flow for the Willmore functional},   Commun. Anal.
Geom. \textbf{10} (2) (2002), 307--39.



\bibitem {KS2} E. Kuwert and R. Sch\"{a}tzle,
    {The Willmore flow with small initial energy}, J. Diff. Geom. \textbf{57} (2001), 409--441.

  \bibitem{KS3}
     E. Kuwert and J. Scheuer, {Asymptotic estimates for the Willmore flow with small
energy},   Int. Math. Res. Notices \textbf{2021} (2021),  14252--14266.




 \bibitem{LSU}
O. A. Ladyzhenskaya, V. A. Solonnikov  and N. N. Uraltseva, {Linear
and Quasilinear Equations of Parabolic Type}, American Math. Society,
1968.

\bibitem{Lan}
J. Langer,  {A compactness theorem for surfaces with $L_p$-bounded second fundamental
form,} Math. Ann. \textbf{270} (1985), 223--34.

 \bibitem{MS}
   J. H. Michael and L. Simon,  {Sobolev and mean-value inequalities on generalized submanifolds of $\mathbb{R}^{n}$}, Comm. Pure Appl. Math.
   \textbf{26} (1973), 361--379.

 \bibitem{Na}   K. Naff,  {A Canonical Neighborhood Theorem for Mean Curvature Flow in Higher Codimension},  Int. Math. Res. Notices \textbf{13} {(2023)}, 11499--11536.


 \bibitem{N}  L. Nirenberg,  {On elliptic partial differential equations}, Ann. Scuola Norm. Sup. Pisa (3) \textbf{13} {1959}, 115--162.


\bibitem{P} 	T. Parker, 	\newblock{Bubble tree convergence for harmonic maps}, 	\newblock{J. Differential Geom.} {\bf 44} (1996),
    595--633.

\bibitem{Q} J. Qing, \newblock{On singularities of the heat flow for harmonic maps from surfaces
  into spheres},  \newblock{ Comm. Anal. Geom.} {\bf 3} (1995), 297--315.


\bibitem{QT} J. Qing and G. Tian, \newblock {Bubbling of the heat flows for harmonic maps from surfaces,} \newblock   {Comm. Pure Appl. Math.}
    {\bf 50} (1997), 295--310.

\bibitem{OC20}
Y. L. Ou and B. Y. Chen,     \newblock{Biharmonic submanifolds and biharmonic maps
in Riemannian geometry}, World Scientific Publishing, Hackensack, NJ,
2020.

\bibitem{PR}F. Palmurella and T. Rivi\`{e}re,
     \newblock{The parametric approach to the Willmore flow},
    \newblock{Adv. Math.} {\bf 400} (2022), 108257.



    \bibitem{S}
  G. Simonett, {The Willmore flow near spheres}, Diff. Int. Equations \textbf{14} (2001), 1005--1014.

  \bibitem{Simon-1968}
J. Simons, Minimal varieties in Riemannian manifolds, Ann. Math.
Second Series {\bf 88} (1968), 62--105.

    \bibitem{st1}
    M. Struwe,
    \newblock{On the evolution of harmonic maps of Riemannian surfaces},
    \newblock{Comm. Math. Helv.} {\bf 60} (1985), 558--581.

    \bibitem{st2}
    M. Struwe,
    \newblock{On the evolution of harmonic maps in higher domensions },
    \newblock{J. Differential Geom.} {\bf 28} (1988), 485--502.

    \bibitem{St3}M.
    Struwe,
    \newblock{Partial regularity for biharmonic maps, revisited},
    \newblock{Calc. Var. $\&$ PDE} {\bf 33} (2008), 249--262.

  \bibitem{W}
   C. Y.  Wang,
    \newblock{Stationary biharmonic maps from $\Bbb R^m$ into a Riemannian manifold},
    \newblock{Comm. Pure Appl. Math.} {\bf 57} (2004), 419--444.



\end{thebibliography}
\end{document}